\documentclass{amsart}

\usepackage[margin=1in]{geometry}
\usepackage{amsthm,amsmath,amsfonts,amssymb}
\usepackage{hyperref}
\usepackage{dsfont}
\usepackage{color}

\usepackage[normalem]{ulem}

\newtheorem{thm}{Theorem}[section]
\newtheorem{prop}[thm]{Proposition}
\newtheorem{cor}[thm]{Corollary}
\newtheorem{ass}[thm]{Assumption}
\newtheorem{lemma}[thm]{Lemma}
\newtheorem{defn}[thm]{Definition}

\theoremstyle{remark}
\newtheorem{remark}[thm]{Remark}
\numberwithin{equation}{section}

\newcommand{\R}{\mathbb R}
\newcommand{\eps}{\varepsilon}
\newcommand{\dd} {\; \mathrm{d}}

\DeclareMathOperator*{\osc}{osc}

\DeclareMathOperator{\diam}{diam}
\DeclareMathOperator{\PV}{PV}

\newcommand{\LL}{\mathcal L}
\newcommand{\K}{\mathcal K}

\title{The Schauder estimate for kinetic integral equations}
\date{\today}

\author{Cyril Imbert} \address[C.~Imbert]{CNRS \& Instituto de
  matem\'atica pura e aplicada, Estrada Dona Castorina, 110, Jardim
  Bot\^anico, CEP 22460-320, Rio de Janeiro, RJ, Brasil}
\email{Cyril.Imbert@ens.fr} \author{Luis Silvestre} \thanks{The
  authors would like to thank C.~Mouhot for fruitful discussions. LS
  is supported in part by NSF grants DMS-1254332 and DMS-1764285.}
\address[L.~Silvestre]{Mathematics Department, University of Chicago,
  Chicago, Illinois 60637, USA} \email{luis@math.uchicago.edu}

\begin{document}
\begin{abstract}
  We establish interior Schauder estimates for kinetic equations with
  integro-differential diffusion. We study equations of the form
  $f_t + v \cdot \nabla_x f = \mathcal L_v f + c$, where
  $\mathcal L_v$ is an integro-differential diffusion operator of
  order $2s$ acting in the $v$-variable. Under suitable ellipticity
  and H\"older continuity conditions on the kernel of $\mathcal L_v$,
  we obtain an a priori estimate for $f$ in a properly scaled H\"older
  space.
\end{abstract}
\maketitle

\section{Introduction}

We study kinetic equations with integral diffusion of the form
\begin{equation} \label{e:main}
 f_t + v \cdot \nabla_x f = \int_{\R^d} (f'-f) K(t,x,v,v') \dd v' + c_0(t,x,v).
\end{equation}
Here, we used the notation from kinetic equations $f=f(t,x,v)$ and
$f'=f(t,x,v')$.  The main result in this article is a
Schauder estimate for equations of the form \eqref{e:main}
whose kernels $K_{(t,x,v)}$ are elliptic and H\"older continuous.

Note that the integral diffusion term on the right hand side of \eqref{e:main} acts on the velocity variable only. The regularization effect on the $x$ variable is a consequence of the interaction between the integral diffusion and the transport term. The equation \eqref{e:main} should be understood as a Kolmogorov-type hypoelliptic integro-differential equation. The diffusion is of fractional order. We work with H\"older-like spaces (given in Definition \ref{d:holder-space}) that are adapted to the particular scaling of this equation in each direction.

Our methods allow us to consider a very general class of kernels $K$. This is essential for the eventual applications of our result to the Boltzmann equation. We start by specifying our notion of ellipticity and H\"older continuity for the kernel $K$. In
\eqref{e:main}, $K(t,x,v,v')$ denotes a function that maps the
variables $(t,x,v)$ into a nonnegative Radon measure $K_{(t,x,v)}$ in $\R^d \setminus \{0\}$:
\[ K_{(t,x,v)} (w) := K(t,x,v,v+w)\]
such that for all $(t,x,v)$, $K_{(t,x,v)}$ belongs to the following ellipticity class of kernels.
\begin{defn}[The ellipticity class] \label{d:class-of-kernels} Given
  the order $2s \in (0,2)$ and ellipticity constants
  $0<\lambda<\Lambda$, we say that a nonnegative Radon measure $K$ in $\R^d \setminus \{0\}$ belongs
  to the ellipticity class $\K$ when the following conditions
  are met. 
\begin{itemize}
\item \textsc{(Symmetry)} $K(w) = K(-w)$.
\item \textsc{(Upper bound)} for all $r>0$
\begin{equation} \label{e:K-upper-bound}
 \int_{B_r} |w|^2 K(w) \dd w \leq \Lambda r^{2-2s}.
\end{equation}
\item \textsc{(Coercivity estimate)} for any $R>0$ and $\varphi \in C^2(B_R)$,
  \begin{equation}
             \label{e:K-coercivity}
         \iint_{B_R \times B_R} |\varphi(v) - \varphi(v')|^2 K(v'-v) \dd v' \dd v \geq \lambda  \iint_{B_{R/2} \times B_{R/2}} |\varphi(v) - \varphi(v')|^2 |v'-v|^{-d-2s} \dd v' \dd v.
\end{equation}
In case $s <(0,1/2)$, we add the following non-degeneracy assumption to the kernel.
\begin{equation} \label{e:K-nondeg}
 \inf_{|e| =1 } \int_{B_r} (w \cdot e)_+^2 K (w) \dd w \ge \lambda r^{2-2s}. 
\end{equation}
\end{itemize}
\end{defn}
\begin{remark}
  Strictly speaking, \eqref{e:K-upper-bound},  \eqref{e:K-coercivity} and \eqref{e:K-nondeg} should be written with integrals on balls
  minus the origin. It is customary to extend the measure $K \dd w$ to
  have zero point mass at the origin. Other choices do not make any
  difference since all our integrands equal zero at $w=0$.
\end{remark}

\begin{remark}
We write $K$ to denote a nonnegative measure on $\R^d$. Even
though we use the notation $K(v') \dd v'$ as if this measure was
absolutely continuous, it does
not need to be. We abuse notation in this way because it makes some
formulas look simpler. For example, we write $K(v'-v)\dd v'$ to denote
the measure in terms of the variables $v'$ and translated by
$v$. Otherwise, for a measure $\mu=K(w) \dd w$, it would typically be
written $\dd ( \tau_{v}\mu(v'))$ which is arguably more confusing.
\end{remark}

\begin{remark}
When $s < 1/2$ we complement the coercivity estimate \eqref{e:K-coercivity} with the non-degeneracy assumption \eqref{e:K-nondeg}. These two assumptions may be redundant. Indeed, for stable-like kernels of the form $K(w) = |w|^{-d-2s} a(w/|w|)$, they are equivalent. This follows easily by computing the Fourier symbol of the operator associated with $K$ (see for example \cite{samorodnitsky1996stable}, and also \cite{MR3482695}). For non-stable-like processes the situation is less clear. We do not know any example of a kernel $K$ satisfying the upper bound \eqref{e:K-upper-bound} that satisfies one of the assumptions in the coercivity estimate but not the other. The non-degeneracy assumption \eqref{e:K-nondeg} is typically much easier to check than the first coercivity estimate \eqref{e:K-coercivity}.
\end{remark}

For local equations, a Schauder estimate refers to an estimate in a
H\"older space when coefficients of the equations are H\"older
continuous.  For non-local equations, the regularity of the
coefficients is replaced with the H\"older continuous dependence of
the kernel with respect to the  variable $z=(t,x,v)$. 
\begin{ass}[H\"older continuity of coefficients] \label{a:K-holder}
Given $z_1 = (t_1,x_1,v_1)$ and $z_2 = (t_2,x_2,v_2)$, for any $r>0$ we have
\[ \int_{B_r} |K_{z_1}(w) - K_{z_2}(w)| |w|^2 \dd w \leq A_0 r^{2-2s} d_\ell(z_1,z_2)^\alpha,\]
where $d_\ell(z_0,z_1)$ stands for the kinetic distance, see Definition \ref{d:distance} below.
\end{ass}
We can now state the Schauder estimate for the equations we consider.
\begin{thm}[The Schauder estimate] \label{t:main} Let
  $0 < \gamma < \min(1,2s)$ and $\alpha = \frac{2s}{1+2s} \gamma$.
  Let $K(t,x,v,v')$ be a kernel such that the two following conditions
  hold true.
\begin{itemize}
\item \textsc{(Ellipticity)} For each $z=(t,x,v) \in Q_1$, the kernel $K_z(w) = K(t,x,v,v+w)$ belongs to the class $\K$ described in Definition \ref{d:class-of-kernels}.
\item \textsc{(H\"older continuity)} Assumption \ref{a:K-holder} holds.
\end{itemize}
If
$f \in C_\ell^\gamma([-1,0] \times B_1 \times \R^d)$ solves \eqref{e:main} in $Q_1$, then the following estimate holds
\[ \|f\|_{C_\ell^{2s+\alpha}(Q_{1/2})} \leq C \left( \|f\|_{C_\ell^\gamma([-1,0] \times B_1 \times \R^d)} + \|c\|_{C_\ell^\alpha(Q_1)} \right).\]
The constant $C$ depends on $d$, $s$, $\lambda$, $\Lambda$, and $A_0$.
\end{thm}

We use the same notation as in \cite{imbert2016weak}: $Q_r$ denotes the kinetic cylinder $Q_r := (-r^{2s},0] \times B_{r^{1+2s}} \times B_r$.

\begin{remark}
  The H\"older norms $C_\ell^{2s+\alpha}$ and $C_\ell^\alpha$ must be
  appropriately understood. They refer to the usual notion of
  $C^{2s+\alpha}$ and $C^\alpha$ regularity with respect to the $v$
  variable. The order of regularity in the other directions is
  adjusted in terms of the invariant structure of the class of
  equations.  On the one hand, the fact that the diffusion is of order
  $2s$ yields an invariant scaling. On the other hand the equation
  enjoys Galilean invariance, yielding a Lie group structure. We
  discuss other choices of distances and their differences in
  Section~\ref{s:discussion}. H\"older spaces are introduced in
  Definition~\ref{d:holder-space} below. The subindex ``$\ell$''
  refers to the fact that the H\"older norm is taken with respect to a
  distance that is left-invariant by the Lie group structure. 
\end{remark}
\begin{remark} \label{rem:alpha} Note that our theorem holds for
  $\alpha = \frac{2s}{1+2s} \gamma  < \gamma$. The distinction between
  these two H\"older exponents $\alpha$ and $\gamma$ comes from
  technical reasons related to the fact that the class of equations is
  left invariant, but not right invariant with respect to the Lie
  group structure.
\end{remark}
\begin{remark}
  Note that the global $C_\ell^\gamma$ norm of $f$ cannot be replaced by
  its global $L^\infty$ norm even in the case of the space-homogeneous
  parabolic fractional heat equation. A solution $f(t,v)$ to
\[ f_t + (-\Delta)_v^s f = 0 \qquad \text{in } Q_1,\]
does {\bf not} satisfy the estimate
\[ \|f\|_{C_\ell^{2s+\alpha}(Q_{1/2})} \leq C \|f\|_{L^\infty([-1,0] \times
    B_1 \times \R^d)}.\] This is because $f$ will not be better than
Lipschitz in time, even though it will have more regularity in
space (See \cite[section~2.4.1]{chang}). The H\"older space
$C_\ell^{2s+\alpha}$ would impose $C^1$ regularity in time for any
$\alpha > 0$.
\end{remark}

\begin{remark}
  Note that the equation \eqref{e:main} does not have a structure
  compatible with the notion of weak solutions in the sense of
  distributions. It is not an equation in divergence form. Our result
  in this paper is an a priori estimate provided that all quantities
  involved make sense classically. It is possible to define a weaker
  notion of solution of \eqref{e:main} in the \emph{viscosity sense},
  and presumably our result in Theorem \ref{t:main} applies to that
  case as well. However, we do not pursue that direction in this paper
  since it would add some technical difficulties obfuscating the
  proofs. The result as currently stated is what we need for our
  intended applications to the Boltzmann equation.
\end{remark}

\begin{remark}
If we want our estimates to
  hold uniformly as $s \to 1$, we would have to replace the constant
  $\lambda$ in \eqref{e:K-coercivity} by $(1-s) \lambda$. The results in this article hold in the local case $s=1$ as well,
  with considerably simplified proofs. It would apply to an equation of the form
\[ f_t + v \cdot \nabla_x f = a_{ij}(t,x,v) \partial_{v_i v_j} f + c(t,x,v).\]
Schauder estimates for these and more general equations have been studied before (see the next subsection). Our approach is quite different to earlier works, starting from the fact that we use a different definition of the H\"older norm. In the case $s=1$, some of the difficulties in the proofs presented here disappear. The ellipticity class in Definition \ref{d:class-of-kernels} would be replaced by the usual uniform ellipticity condition of the coefficients $a_{ij}$. The Assumption \ref{a:K-holder} would translate as the H\"older continuity assumption of these coefficients. The section about weak limits of kernels would be unnecessary since it would be replaced by the simple convergence of the matrix of coefficients. The majorant function defined in \eqref{e:majorant}, that plays an important role later in the proof of the Liouville theorem, would be irrelevant since the equation is local. Consequently, the final estimate would be in terms of $\|f\|_{L^\infty(Q_1)}$ instead of $\|f\|_{C^\gamma((-1,0] \times B_1 \times \R^d)}$. Moreover, the equation \eqref{e:L-eq-for-g} in the Liouville theorem could be written in terms of $f$ directly, instead of introducing the function $g$. The final result is an interior estimate for the $C^{2+\alpha}_\ell$ norm with $s=1$. This norm is comparable but contains more explicit information than the norms used previously in the literature (see the norm $H^\alpha$ in \cite{imbert2018toy} for example, or equivalently the norm $\|\cdot\|_{2+\alpha,\Omega}$ in \cite{sergio2004recent}). Indeed, the inequality $\|f\|_{H^\alpha} \lesssim \|f\|_{C^{2+\alpha}_\ell}$ follows easily from Lemma \ref{l:holder-derivatives}. The inequality in the opposite direction is much less obvious. A posteriori, one can get it (away from the boundary) as a consequence of our Schauder estimates. Indeed, if $f \in H^\alpha$, then it is clear that $f_t + v \cdot \nabla_x f - \Delta_v f \in C_\ell^\alpha$. Our Schauder estimate for the case $s=1$ tells us that
\[ [f]_{C_\ell^{2+\alpha}(Q_{1/2})} \lesssim \|f\|_{L^\infty(Q_1)} + \| f_t + v \cdot \nabla_x f - \Delta_v f \|_{C^\alpha(Q_1)} \leq \|f\|_{H^\alpha(Q_1)}.\] 
For $s<1$, it is not in general possible to redefine the space $C^{2s+\alpha}_\ell$ in terms of H\"older norms of derivatives as in the classical definition of $H^\alpha$.
\end{remark}

\subsection{Schauder estimates for kinetic and non-local equations}

Linear kinetic equations of second order are a particular instance of
the more general theory of ultraparabolic equations of Kolmogorov
type. Results involving regularity estimates in H\"older spaces for
these equations appeared especially in the late 1990's. See
\cite{sayro1971,manfredini, lunardi1997, eidelman1998,
  dfp,radkevich2008}, and the survey article
\cite{sergio2004recent}. More recently in
\cite{henderson2017c,imbert2018toy}, Schauder estimates were applied
to bootstrap higher regularity estimates for second order models in
kinetic theory, including the Landau equation with moderately soft
potentials. The Boltzmann equation can be written in the form
\eqref{e:main} for a kernel $K$ depending on the solution itself (see
\cite{luis-cmp}). It is our intention to use the result of
Theorem~\ref{t:main} to derive higher order regularity estimates for
the non-cutoff Boltzmann equation in (shortly) forthcoming work.

Schauder estimates for integro-differential equations have been
obtained in recent years, see
\cite{MP,tj2015,serra1,joaquim2015,imbert2016schauder,tj2018}. They have the well known
difficulty that the smoothness of the tails of the integrals outside
of the domain of the equation are difficult to control. There are two
common workarounds that have been used in the literature. One
workaround that works well in the elliptic case is to impose some
regularity in the kernels $K$ with respect to the variable of
integration $w$. Another approach, arguably more delicate, imposes
extra regularity on the values of the function $f$ outside of the
domain. That is the case in \cite{serra1},
\cite{joaquim2015},\cite{tj2018} and it is also the approach we take
here. In our kinetic setting, this restriction goes a bit further by
requiring the function $f \in C_\ell^\gamma$ outside of $Q_1$, with
$\gamma > \alpha$. 

In this paper, we use some key ideas that originated in \cite{serra1} and simplify enormously the general procedure to prove the Schauder estimates in the nonlocal setting. The key of the proof of Theorem \ref{t:main} is a combination of a blow-up technique (see Proposition \ref{p:blowup}) with a Liouville theorem (see Theorem \ref{t:liouville}).

\subsection{Possible extensions and outstanding questions}

It is most probably possible to extend Theorem~\ref{t:main} to
  higher values of $\gamma$ such that
  $\gamma \notin \mathbb N + 2s \mathbb N$. It would require the
  extension of Assumption~\ref{a:K-holder} for $\alpha$ large, or a
  version of \eqref{e:L-eq-for-g} involving higher order incremental
  quotients. It would not be obvious how to imply the result of
  Theorem \ref{t:main} for higher values of $\gamma$ by simply taking
  derivatives. Knowing $f \in C_\ell^\gamma$ gives us a clean estimate for
  $f_t + v \cdot \nabla_x f$ in $C_\ell^{\gamma-2s}$ and $\partial_{v_i} f$
  in $C_\ell^{\gamma-1}$. These two derivatives do not solve an equation
  like \eqref{e:main}. One might attempt to apply hypoelliptic estimates
  to derive corresponding H\"older spaces for $f_t$ and
  $\partial_{v_i}f+v_i \partial_{x_i} f$. But then we would loose a
  fraction of the H\"older exponent that goes beyond the order of
  differentiation. This is somehow reflected in the final relation
  between $\alpha$ and $\gamma$ in the statement of
  Theorem~\ref{t:main}.  \smallskip

  In the statement of Theorem~\ref{t:main}, we make the assumption
  $\alpha \le \frac{2s}{1+2s} \gamma$. We know that the theorem would not
  hold with $\alpha > \gamma$. The range
  $\frac{2s}{1+2s} \gamma  < \alpha \leq \gamma$ is currently unclear.

\subsection{Organization of the article}
The article is organized as follows. In Section~\ref{s:holder},
H\"older spaces adapted to the study of kinetic equations are
introduced. In particular, a kinetic degree of polynomials and
differential operators and a kinetic distance are
defined. Section~\ref{s:integral} is devoted to the study of integral
operators associated with the class $\K$ of elliptic kernels from
Definition~\ref{d:class-of-kernels}. We then state and prove a
Liouville type theorem in Section~\ref{s:liouville}. The final
section~\ref{s:blowup} is devoted to the proof of the main theorem. It
is done by contradiction through a blowup argument.

\section{Kinetic H\"older spaces}
\label{s:holder}

In this preliminary section, we mainly introduce the H\"older spaces
we need to derive the Schauder estimate for kinetic
integro-differential equations. We first define a kinetic distance,
then the kinetic degrees of polynomials and differential
operators. The definition of H\"older spaces is then given and an
interpolation inequality is proved.

\subsection{The kinetic distance}
\label{ss:kd}

The following Lie group structure of $\R^{1+2d}$ plays a key role in all
our computations. The product is defined
as
\[
  (t_1,x_1,v_1) \circ (t_2,x_2,v_2) = (t_1+t_2, x_1 + x_2 + t_2 v_1,  v_1+v_2).
\]
Note that this product is not commutative. The class
equations we will be working with (as in \eqref{e:kinetic-constant})
are left-invariant, in the sense that if $f(z)$ is a solution of \eqref{e:main}, then
$f_0(z) := f(z_0 \circ z)$ is also a solution of a similar equation with
a translated right hand side and a translated kernel in the same ellipticity class.

There is also invariance by scaling. We define
$S_R(t,x,v) = (R^{2s} t, R^{1+2s} x, Rv)$. If $f(z)$ solves an
equation like \eqref{e:main}, then $f(S_R z)$ solves a similar
equation with a scaled right hand side and a scaled kernel in the same ellipticity class.

Because of this property, it is good to work with a notion of
distance, H\"older norms, degree and differential operators, that are
homogeneous respect to this kinetic scaling, and are left invariant by
the action of the Lie group.

\begin{defn}[A left-invariant distance]\label{d:distance}
Given two points $z_1 = (t_1,x_1,v_1)$ and $z_2 = (t_2, x_2, v_2)$ in $\R^{1+2d}$, we define the following distance function
\[ d_\ell(z_1,z_2) := \min_{w \in \R^d} \left\{ \max \left( |t_1-t_2|^{ \frac 1 {2s} } , |x_1-x_2-(t_1-t_2)w|^{ \frac 1 {1+2s} } , |v_1-w| , |v_2-w| \right) \right\}.\]
\end{defn}

The subindex ``$\ell$'' stands for \textbf{``l''}eft invariant.

It is convenient to also have a notion of ``norm'' with the right scale invariance. We define: 
\begin{equation} \label{e:norm}
\|(t,x,v)\| = \max \left\{ |t|^{1/(2s)} , |x|^{1/(1+2s)} , |v| \right\}.
\end{equation}
Note that it is \textbf{not} an actual norm in the strictly mathematical sense of the term.

Here are some observations.
\begin{itemize}
\item The distance $d_\ell$ is left invariant by the Lie group action in the sense that $d_\ell(z \circ z_1, z \circ z_2) = d_\ell(z_1, z_2)$ for any $z,z_1,z_2 \in \R^{1+2d}$.
\item It is homogeneous with respect to scaling: $d_\ell(S_R z_1, S_R z_2) = R d_\ell(z_1,z_2)$.
\item  We will see below in Proposition \ref{p:distance} that $d_\ell$ is indeed a distance when $s \geq 1/2$ in the sense that it satisfies the triangle inequality. When $s<1/2$, the function $d_\ell^{2s}$ is a distance. We will still work with $d_\ell$ (as opposed to $d_\ell^{2s}$) when $s <1/2$ so that we keep a consistent scaling formula (as in the previous bullet point) throughout the paper.
\item There are other equivalent formulas to measure how far apart $z_1$ and $z_2$ are. We observe that
\begin{align*} 
d_\ell (z_1,z_2) &\approx \| z_2^{-1} \circ z_1 \| , \\ 
 &\approx \| z_1^{-1} \circ z_2 \| , \\ 
 &\approx \inf_{w \in \R^d} |t_2-t_1|^{1/(2s)} + |x_1-x_2+(t_1-t_2)w|^{1/(1+2s)} + |v_1-w| + |v_2-w|.
\end{align*}
None of the three formulas on the right hand side are proper distances. However, since they give us a good estimate for $d_\ell(z_1,z_2)$ we will use them whenever it is convenient.
\item Note that the distance $d_\ell$ can be reformulated in the following way: $d_\ell(z_1,z_2)$ is the infimum value of $r > 0$ so that both $z_1$ and $z_2$ belong to a cylinder $Q_r(z)$ for some $z \in \R^{1+2d}$.
\item Our usual definition for the cylinder $Q_r$ would not be affected significantly if we changed it for
\[ Q_r = \{(t,x,v): t \leq 0 \text{ and } d_\ell(0, (t,x,v)) < r\}.\]
Moreover, because of the Lie group invariance, we could also have for $z_0 = (t_0,x_0,v_0)$,
\[ Q_r(z_0) = \{z=(t,x,v): t \leq t_0 \text{ and } d_\ell(z_0, z) < r\}.\]
\end{itemize}

\begin{prop} \label{p:distance}
The function $d_\ell : \R^{1+2d} \times \R^{1+2d} \to [0,\infty)$ is a distance when $s \geq 1/2$. For $s < 1/2$, the function $d_\ell^{2s}$ is a distance.
\end{prop}
\begin{proof}
  We start with the case $s \geq 1/2$. Because of the invariance by
  the Lie group action, we only need to prove the triangle inequality
  when one of the three points is the origin. That is, given any
  $z_1 , z_2 \in \R^{1+2d}$, we must show that
  $d_\ell(z_1,0) + d_\ell(0,z_2) \geq d_\ell(z_1,z_2)$.

Let $w_1, w_2 \in \R^d$ be the points where the minimum in Definition \ref{d:distance} is achieved for $z_1$ and $z_2$ respectively. That is
\begin{align*}
d_\ell(z_1,0) = \max\left( |t_1|^{\frac 1 {2s}} , |x_1 + t_1 w_1|^{\frac 1 {1+2s}} , |v_1-w_1|, |w_1| \right), \\
d_\ell(z_2,0) = \max\left( |t_2|^{\frac 1 {2s}} , |x_2 + t_2 w_2|^{\frac 1 {1+2s}} , |v_2-w_2|, |w_2| \right).
\end{align*}

By definition, $d_\ell(z_1,z_2)$ is the minimum over all choices of $w \in \R^d$, so it is less or equal to the value we get by setting $w=w_1+w_2$.
\begin{align*}
d_\ell(z_1,z_2) &\leq  \max \left( |t_1-t_2|^{ \frac 1 {2s} } , |x_1-x_2 + (t_1-t_2)(w_1+w_2)|^{ \frac 1 {1+2s} } , |v_1-w_1-w_2| , |v_2-w_1-w_2| \right).
\end{align*}
We now analyze every one of the four expressions inside the $\max$.

Clearly, since $1/(2s) \leq 1$, we have
\[ |t_1-t_2|^{ \frac 1 {2s} } \leq |t_1|^{ \frac 1 {2s} } + |t_2|^{ \frac 1 {2s} } \leq d_\ell(z_1,0) + d_\ell(z_2,0).\]
Also, simply by the triangle inequality in $\R^d$,
\[ |v_1-w_1-w_2| \leq |v_1-w_1| + |w_2| \leq d_\ell(z_1,0) + d_\ell(z_2,0).\]
Likewise,
\[ |v_2-w_1-w_2| \leq |v_2-w_2| + |w_1| \leq d_\ell(z_2,0) + d_\ell(z_1,0).\]
The second argument in the max is the only one that requires a nontrivial analysis. We evaluate
\begin{align*} 
 |x_1-x_2-(t_1-t_2)(w_1+w_2)|^{ \frac 1 {1+2s} } &= |(x_1+t_1 w_1)- (x_2 + t_ 2w_2) + t_1 w_2 - t_2 w_1|^{ \frac 1 {1+2s} }, \\
&\leq |d_\ell(z_1,0)^{1+2s} + d_\ell(z_2,0)^{1+2s} \\
& \qquad + d_\ell(z_1,0)^{2s} d_\ell(z_2,0) + d_\ell(z_2,0)^{2s} d_\ell(z_1,0) |^{ \frac 1 {1+2s} }, \\
& \leq d_\ell(z_1,0) + d_\ell(z_2,0).
\end{align*}
The last inequality follows from the following elementary calculus fact. For any $a,b \geq 0$ and $p \geq 1$, the following inequality holds:
\begin{equation} \label{e:calc-fact}
 a^{1+p} + b^{1+p} + a^p b + a b^p \leq (a+b)^{1+p}.
\end{equation}
Clearly, in the last inequality, we applied \eqref{e:calc-fact} with $a= d_\ell(z_1,0)$, $b=d_\ell(z_2,0)$ and $p=2s$.

The proof for the case $s < 1/2$ goes along the same lines, and we conclude the last inequality also applying \eqref{e:calc-fact} with $p=1/(2s)$.
\end{proof}

\subsection{Kinetic degree of polynomials}

We start by defining a modified notion of degree for a polynomial
$p \in \R[t,x,v]$. This special degree, which we will call
\emph{kinetic degree}, matches the scaling of the equation.

In order to compute $\deg_k p$, every exponent of the variable $t$
should count times $2s$, every exponent of the variables $x_i$ counts
times $(1+2s)$ and the exponents of the variables $v_i$ count
normally. More precisely, if $m \in \R[t,x,v]$ is a monomial, we
define its kinetic degree $\deg_k m$ is the number $\kappa$ so that
$m(S_R z) = R^\kappa m(z)$.

A polynomial $p$ is always a finite sum of monomials. In general, we
define the kinetic degree of a polynomial $p = \sum m_j$ as the
maximum of $\deg_k m_j$ for all its monomial terms $m_j$.

Note that the degree of a polynomial $p \in \R[t,x,v]$ can be any
number in the discrete set $\mathbb N + 2s \mathbb N$.

\subsection{H\"older spaces}

We now define a properly scaled version of H\"older spaces.
\begin{defn}[H\"older spaces] \label{d:holder-space} For any
  $\alpha \in (0,\infty)$, we say that a function $f: D \to \R$ is
  $C_\ell^\alpha$ at a point $z_0 \in \R^{1+2d}$ if there exists a
  polynomial $p \in \R[t,x,v]$ such that $\deg_k p < \alpha$ and for
  any $z \in D$
  \[ |f(z) - p(z)| \leq C d_\ell(z,z_0)^\alpha.\] When this property
  holds at every point $z_0$ in the domain $D$, with a uniform
  constant $C$, we say $f \in C_\ell^\alpha(D)$. The semi-norm
  $[f]_{C_\ell^\alpha(D)}$ is the smallest value of the constant $C$ so
  that the inequality above holds for all $z_0, z \in D$. The norm
  $\|f\|_{C_\ell^\alpha(D)}$ is $[f]_{C_\ell^\alpha(D)}+[f]_{L^\infty(D)}$.
\end{defn}

\begin{remark}
Note that $d_\ell$, $\deg_k$, and therefore also $C^\alpha_\ell$, depend on $s$ implicitly.
\end{remark}

\begin{remark}
  With the above definition, when
  $\alpha \in \mathbb N + 2s \mathbb N$, the $C_\ell^\alpha$ space
  corresponds to a Lipschitz-type space instead of the classical $C^1$
  space.
\end{remark}

Using the invariance by left translations, we can rephrase the
H\"older regularity of $f$ at $z_0$ in the following way. There exists
a polynomial $p_0$ such that $\deg_k p_0 < \alpha$ and for all
$z \in \R^{1+2d}$ such that $z_0 z \in D$, we have
\begin{equation} \label{e:Holder-ineq}
 |f(z_0 \circ  z) - p_0(z)| \leq C \|z\|^\alpha.
\end{equation}
In this case $p_0(z) = p(z_0 \circ z)$ where $p$ is the polynomial of
Definition \ref{d:holder-space}. If the polynomial $p_0$ is given by
\[
  p_0(t,x,v) = a_0 + a_1 t + a_2 \cdot x + a_3\cdot v + \dots,
\]
it is easy to verify that $a_0 = f(z_0)$,
$a_1 = \partial_t f + v_0 \cdot \nabla_x f$, $a_2 = \nabla_x f$ and
$a_3 = \nabla_v f$.

As it is standard for some proofs of the Schauder estimates (see for
example Section 4 in \cite{gilbarg2015elliptic}), we define the
adimensional H\"older spaces.

\begin{defn}[Adimensional H\"older spaces] \label{d:adimensional-Holder-space}
Given a kinetic cylinder $Q \subset \R^{1+2d}$ and any $\alpha > 0$, we define
\[ [f]_{C_{\ell,\ast}^\alpha(Q)} = \sup_{z \in Q} d_z^\alpha [f]_{C_\ell^\alpha(Q_{d_z}(z))}.\]
where
\begin{equation}\label{e:dz}
  d_z= d_\ell(\partial_p Q,z).
\end{equation}
Naturally, we also define $\|f\|_{C_{\ell,\ast}^\alpha(Q)} := \|f\|_{C_\ell^0(Q)} + [f]_{C_{\ell,\ast}^\alpha(Q)}$.
\end{defn}

Here, we used the notation $\partial_p Q_r(t_0,x_0,v_0) := \partial Q_r(t_0,x_0,v_0) \setminus \{t=t_0\}$.

\subsection{H\"older norms and differential operators}

The differential operators $(\partial_t + v \cdot \nabla_x)$,
$\partial_{x_i}$ and $\partial_{v_i}$ commute with left
translations. They do not commute with each other, and they do not
keep the equation \eqref{e:kinetic-constant} invariant (with
$c=0$). The operators that commute with the equation
\eqref{e:kinetic-constant} are $\partial_t$, $\partial_{x_i}$ and
$\partial_{v_i} + v_i \partial_{x_i}$, which are the ones that commute
with right translations (instead of left translations).

It may be convenient to define the \emph{kinetic degree} of a differential
operator. We say that the kinetic degree of
$\partial_t + v \cdot \nabla_x$ is $2s$, the kinetic degree of
$\partial_{x_i}$ is $(1+2s)$ and the kinetic degree of
$\partial_{v_i}$ is $1$.


It is convenient to relate the definition of the H\"older spaces
$C_\ell^\alpha$ with these operators. The following (deceivingly simple)
lemma will be used repeatedly.
\begin{lemma}[Derivatives and kinetic H\"older spaces] \label{l:holder-derivatives}
  Let $D=\partial_t + v \cdot \nabla_x$, $D=\partial_{x_i}$ or
  $D = \partial_{v_i}$. Let $f$ be a $C_\ell^\alpha$ function in a cylinder $Q$ and let $\deg_k D = \kappa$ with $\kappa < \alpha$. Then
  $Df \in C_\ell^{\alpha-\kappa}$ and
\[ [Df]_{C_\ell^{\alpha-\kappa}(Q)} \lesssim [f]_{C_\ell^\alpha(Q)}.\]
\end{lemma}
Before proving Lemma \ref{l:holder-derivatives}, we prove the following auxiliary lemma about polynomials.
\begin{lemma} \label{l:polynomial-norm}
Let $p(z)$ be a polynomial in $t$, $x$ and $v$ of kinetic degree $k$. Let us write $p$ as
\[ p(t,x,v) = \sum_{j \in \mathbb N^{1+2d}}  a_j m_j(z),\]
where
\[ m_j(z) := t^{j_0} x_1^{j_1} \dots x_d^{j_d} v_1^{j_{d+1}} \dots v_d^{j_{2d}}.\]
Assume that
\[ \sup_{\|z\| \leq r} |p(z)| \leq C_0 r^\alpha.\]
Then, for each $j$, we have
\[ |a_j| \leq C C_0 r^{\alpha - \deg_k m_j},\]
where the constant $C$ depends on $\deg_k p$ and dimension only.
\end{lemma}
\begin{proof}
We observe that once we establish this lemma for $r=1$, the other values of $r$ follow by scaling.

The space of polynomials of kinetic degree $k$ in $\R^{1+2d}$ is finite dimensional. Recall that all norms are equivalent in spaces of finite dimension. The result for $r=1$ follows easily by comparing the two norms given by
\begin{equation}
  \label{e:p-norm}
  \|p\|_1 = \sup_j |a_j| \qquad \text{ and } \qquad \|p\|_2 = \sup_{\|z\| \leq 1} |p(z)|.
\end{equation}
This concludes the proof of the technical lemma. 
\end{proof}

\begin{proof}[Proof of Lemma \ref{l:holder-derivatives}]
  Let $z_0$ and $z_1$ be two points in $Q$. Since $f \in C_\ell^\alpha(Q)$,
  there exist polynomials $q_0$ and $q_1$ of degree less than $\alpha$
  so that for all $z$ so that $z_0 \circ z \in Q$ and
  $z_1 \circ z \in Q$,
\begin{align*} 
 |f(z_0 \circ  z) - q_0(z)| &\lesssim C \|z\|^\alpha, \\
 |f(z_1 \circ  z) - q_1(z)| &\lesssim C \|z\|^\alpha
\end{align*}
where $C = [f]_{C_\ell^\alpha}$. We used the fact that $d_\ell$ is left invariant and $d_\ell (z,0) \le \|z\|$.
Let $r = d_\ell(z_0,z_1) \approx \|z_0^{-1} \circ z_1\|$ and let us pick any $z$ so that $\|z\| < r$ \footnote{It is slightly problematic when $z_0 \circ z$ or $z_1 \circ z$ fall outside the domain $Q$. It can be handled like in the proof of Proposition \ref{p:adimensional-interpolation} by using the equivalent norm in the space of polynomials of degree $< \alpha$ that considers only the points that fall inside the domain.}. From the triangle inequality (modified with the power $2s$ when $s<1/2$), we have $\|z_0^{-1} \circ z_1 \circ  z\| \approx d_\ell(z_0,z_1\circ z) \leq d_\ell(z_0,z_1) + d_\ell(z_1,z_1\circ z) \lesssim r$. We apply the two inequalities above for $z_0^{-1} \circ z_1 \circ z$ and $z$ respectively to obtain
\begin{align*} 
 |f(z_1 \circ  z) - q_0(z_0^{-1} \circ z_1 \circ z)| &\lesssim C r^\alpha, \\
 |f(z_1 \circ  z) - q_1(z)| &\lesssim C r^\alpha.
\end{align*}
Therefore, for all $\|z\| \leq r$, 
\[ |q_1(z) - q_0(z_0^{-1} \circ z_1 \circ z)| \lesssim C r^\alpha.\]
 Let us write the coefficients of both polynomials $q_1(z)$ and $q_0(z_0^{-1} \circ z_1 \circ z)$.
\begin{align*} 
  q_0(z_0^{-1} \circ z_1 \circ  z) &= a_0(z_1) + a_1(z_1) t
                                     + a_2(z_1) x + a_3(z_1) v + \dots,  \\
q_1(z) &= b_0 + b_1 t + b_2 x + b_3 v + \dots
\end{align*}
The coefficients $a_j(z_1)$ can be computed in terms of the coefficients of $q_0$ and the value of $z_0^{-1} \circ z_1$. It is not hard to see that each coefficient is a polynomial in $z_0^{-1} \circ z_1$ and thus of  $z_1$ whose degree is not larger than $\deg_k q_0$ minus the degree of the corresponding monomial. Applying Lemma \ref{l:polynomial-norm}, we see that
\[ | b_j - a_j(z_1) | \leq C r^{\alpha - k_j}.\] Here
$k_1 = \deg_k t = \deg_k (\partial_t+ v \cdot \nabla_x) = 2s$,
$k_2 = \deg_k x =\deg_k \partial_{x_i} = 1+2s$ and $k_3= \deg_k v =\deg_k \partial_{v_i} = 1$, i.e. $k_j = \kappa$ in short. Since
$b_1 = (\partial_t+v_1 \nabla_x) f(z_1)$, $b_2 = \nabla_x f(z_1)$ and
$b_3 = \nabla_v f(z_1)$, i.e. $b_j = Df(z_1)$ in short, we have
\[ \sup \{z_1  | Df (z_1) - a_j(z_1) | : d_\ell (z_0,z_1) < r\} \leq C r^{\alpha - \kappa}.\]
This finishes the proof. 
\end{proof}
\begin{remark}
  Note we can also apply Lemma \ref{l:polynomial-norm} to the other coefficients of the polynomial $q_1(z)$. If $p_z$ is the polynomial expansion of $f$ at $z$ of kinetic degree less than $\alpha$ and it has the form
\[ p_z(t,x,v) = \sum_j a_j(z) t^{-j_0} x_1^{j_1} \dots x_d^{j_d} v_1^{j_{d+1}} \dots v_d^{j_{2d}},\]
then, by a direct computation, the coefficients $a_j(z)$ correspond to
\[ a_j(z) = \frac{(\partial_t + v \cdot \nabla_x)^{j_0} \partial_{x_1}^{j_1} \dots \partial_{x_d}^{j_d} \partial_{v_1}^{j_{d+1}} \dots \partial_{v_d}^{j_{2d}} f(z)}{k_0! \dots k_{2d}!}.\]
Note that $(\partial_t+v \cdot \nabla_x)$ and $\partial_v$ do not commute.
 \end{remark}

\subsection{Interpolation inequalities}

The usual interpolation estimates for H\"older spaces will hold.
\begin{prop}[Interpolation] \label{p:adimensional-interpolation}
Given $\alpha_1 < \alpha_2 < \alpha_3$ so that $\alpha_2 = \theta \alpha_1 + (1-\theta) \alpha_3$, then for any function $f \in C_\ell^{\alpha_3}(Q_1)$,
\[ [f]_{C_\ell^{\alpha_2}(Q_1)} \leq [f]_{C_\ell^{\alpha_1}(Q_1)}^\theta [f]_{C_\ell^{\alpha_3}(Q_1)}^{1-\theta} + [f]_{C_\ell^{\alpha_1}(Q_1)}.\]
Also, for any function $f \in C_{\ell,\ast}^{\alpha_3}(Q_1)$,
\[ [f]_{C_{\ell,\ast}^{\alpha_2}(Q_1)} \leq [f]_{C_{\ell,\ast}^{\alpha_1}(Q_1)}^\theta [f]_{C_{\ell,\ast}^{\alpha_3}(Q_1)}^{1-\theta} + [f]_{C_{\ell,\ast}^{\alpha_1} (Q_1)}.\]
\end{prop}
\begin{remark}\label{r:inter-eps}
  We classically get from the previous estimates that for all $\eps >0$,
  \begin{align*}
    [f]_{C_\ell^{\alpha_2}(Q_1)} &\lesssim C_\eps [f]_{C_\ell^{\alpha_1}(Q_1)}+ \eps [f]_{C_\ell^{\alpha_3}(Q_1)},\\
    [f]_{C_{\ell,\ast}^{\alpha_2}(Q_1)} &\lesssim C_{\ast,\eps} [f]_{C_{\ell,\ast}^{\alpha_1}(Q_1)} + \eps [f]_{C_{\ell,\ast}^{\alpha_3}(Q_1)}.
  \end{align*}
\end{remark}
\begin{proof}
We prove the first interpolation inequality. The second one follows as a consequence by scaling. The statement says precisely that the function $\alpha \mapsto \log [f]_{C_\ell^\alpha}$ is convex. This is a local property, so we only need to prove it for $\alpha_3$ sufficiently close to $\alpha_1$. Because of this, it is enough to prove the interpolation inequality assuming that $(\mathbb N + 2s \mathbb N) \cap [\alpha_1, \alpha_3)$ contains at most one element.

Let $q^1_z$, $q^2_z$ and $q^3_z$ be the polynomial expansions of $f$ at $z$ of kinetic degrees less than $\alpha_1$, $\alpha_2$ and $\alpha_3$ respectively such that for all $z \circ \xi \in Q_1$, 
\begin{equation} \label{e:hh1}
 |f(z \circ \xi) - q^i_z(\xi))| \leq [f]_{C_\ell^{\alpha_i}} \|\xi\|^{\alpha_i} \qquad \text{ for } i =1,2,3.
\end{equation}
The polynomials $q^1_z$, $q^2_z$ and $q^3_z$ are increasingly higher order expansions at the same point $z$. Therefore, $q^2_z$ contains all the terms in $q^1_z$ plus perhaps higher order ones. In the same way, $q^3_z$ contains all the terms of $q^2_z$ plus perhaps higher order ones. Because of our assumption that $(\mathbb N + 2s \mathbb N) \cap [\alpha_1, \alpha_3)$ has at most one element,  which we call $\bar \alpha$, there can be at most one degree of homogeneity in the difference between the polynomials $q_z^1$ and $q_z^3$. The polynomial $q^2_z$ coincides with either $q^1_z$ or $q^3_z$ depending on whether $\bar \alpha \geq \alpha_2$ or $\bar \alpha < \alpha_2$. When $(\mathbb N + 2s \mathbb N) \cap [\alpha_1, \alpha_3) = \emptyset$, we have  $q_z^1 = q^2_z = q^3_z$ that is easier to analyze. Let us consider the case in which there is an $\bar \alpha$ and the polynomials are not equal.

Like in Lemma \ref{l:polynomial-norm}, we write
\[ q^i_z(\xi) = \sum_{|j| < \alpha_i} a_j(z) m_j(\xi).\]
where each $m_j$ is a monomial.

Substracting \eqref{e:hh1} for $i=1,2$, whenever $z \circ \xi \in Q_1$, we have
\begin{equation} \label{e:q3-q1}
 \left \vert q_z^3(\xi) - q_z^1(\xi) \right\vert = \left \vert \sum_{|j| = \bar \alpha} a_j(z) m_j(\xi) \right\vert \leq  [f]_{C_\ell^{\alpha_1}} \|\xi\|^{\alpha_1}  +  [f]_{C_\ell^{\alpha_3}} \|\xi\|^{\alpha_3} .
\end{equation}
From this inequality, we want to infer an estimate for $\|q_z^3 - q_z^1\|$. Let us first make some remarks about the norm of a polynomial. The space of polynomials of kinetic degree less than $\alpha_3$ is finite dimensional. So, all norms that we can write are equivalent. A natural choice is perhaps
\[ \|q\| = \max \{ |q(\xi)| : \|\xi\| \leq 1 \}.\]
If we change that radius $1$ for any other universal constant, we would obtain an equivalent norm. Note that translations of polynomials are also polynomials of the same degree. Therefore, for any two universal constants $c$ and $C$, $z_0 \in \R^{1+2d}$, and $\deg_k q < \alpha_3$, we deduce that
\[ \max \{ |q(\xi)| : \|\xi\| \leq Q_c(z_1) \} \approx \max \{ |q(\xi)| : \|\xi\| \leq Q_C(z_1) \}.\]
The factors in $\approx$ depend naturally on $c$ and $C$.

Coming back to \eqref{e:q3-q1}, for any $N \in (0,1]$ and $z \in Q_1$, let us pick some point $\xi_1 \in Q_1$ such that
\begin{itemize}
\item $\|\xi_1\| < N$,
\item Whenever $d_\ell(\xi_1,\xi)<cN$, then $\|\xi\| \leq N$ and $z \circ \xi \in Q_1$. Here $c$ is a universal constant.
\end{itemize}
It is not hard to see that for any $z \in Q_1$, such $\xi_1$ exists (in fact plenty).

From \eqref{e:q3-q1} we get,
\[ \sup_{ \{\xi : d(\xi_1,\xi) < cN\} } |(q_z^3 - q_z^1) (\xi)| \leq  [f]_{C_\ell^{\alpha_1}} N^{\alpha_1}  +  [f]_{C_\ell^{\alpha_3}} N^{\alpha_3}.\]
Since $q_z^3-q_z^1$ is homogeneous of degree $\bar \alpha$,
\[ \sup_{ \{\xi : d(S_{1/N} \xi_1,\xi) < c\} } |(q_z^3 - q_z^1) (\xi)| \leq  [f]_{C_\ell^{\alpha_1}} N^{\alpha_1-\bar \alpha}  +  [f]_{C_\ell^{\alpha_3}} N^{\alpha_3- \bar \alpha}.\]
Since $\| \xi_1\| \leq N$, then $\|S_{1/N} \xi_1\| \leq 1$. From the triangle inequality (modified with power $2s$ when $s<1/2$), there is a universal constant $C$ so that whenever $\| \xi \| \leq 1$, then $\xi \in d_\ell (S_{1/N} \xi_1, \xi) \leq C$. 

According to the discussion above, the fact that all norms are equivalent in the space of polynomials implies that
\begin{align*} 
\|q_z^3 - q_z^1 \| &=  \sup_{ \{\xi : \| \xi \| \leq 1\} } |(q_z^3 - q_z^1) (\xi)| , \\ 
 &\leq \sup_{ \{\xi : d(S_{1/N} \xi_1,\xi) < C\} } |(q_z^3 - q_z^1) (\xi)|, \\ &\lesssim  [f]_{C_\ell^{\alpha_1}} N^{\alpha_1-\bar \alpha}  +  [f]_{C_\ell^{\alpha_3}} N^{\alpha_3- \bar \alpha}.
\end{align*}

We now optimize for $N \in [0,1]$ and obtain
\[ \|q^3_z - q^1_z\|  \lesssim [f]_{C_\ell^{\alpha_1}}^{\bar \theta} [f]_{C_\ell^{\alpha_3}}^{1-\bar \theta} + [f]_{C_\ell^{\alpha_1}},\]
where $\bar \alpha = \bar \theta \alpha_1 + (1-\bar \theta) \alpha_2$.

Now we estimate $f(z \circ \xi) - q_z^2(\xi)$ using both $[f]_{C_\ell^{\alpha_1}}$ and $[f]_{C_\ell^{\alpha_3}}$. There are two cases depending on whether $\alpha_2 > \bar \alpha$ or $\alpha_2 \leq \bar \alpha$. The proofs are very similar, so let us do only the later. In this case $q^2_z = q^1_z$. We have
\[
|f(z \circ \xi) - q_z^2(\xi)| \leq \begin{cases} [f]_{C_\ell^{\alpha_1}} \|\xi\|^{\alpha_1}, \\
[f]_{C_\ell^{\alpha_3}} \|\xi\|^{\alpha_3} + \left( [f]_{C_\ell^{\alpha_1}}^{\bar \theta} [f]_{C_\ell^{\alpha_3}}^{1-\bar \theta} + [f]_{C_\ell^{\alpha_1}} \right)\|\xi\|^{\bar \alpha}.
\end{cases} \]
One can easily verify that the right hand side is less than $[f]_{C_\ell^{\alpha_1}}^{\theta} [f]_{C_\ell^{\alpha_3}}^{1-\theta} \|\xi\|^{\alpha_2}$ for any value of $\|\xi\|$. Indeed, if $\|\xi\| \geq ([f]_{C_\ell^{\alpha_3}} / [f]_{C_\ell^{\alpha_2}})^{1/(\alpha_3-\alpha_1)}$, we have
\[ [f]_{C_\ell^{\alpha_1}} \|\xi\|^{\alpha_1} \leq [f]_{C_\ell^{\alpha_1}}^{\theta} [f]_{C_\ell^{\alpha_3}}^{1-\theta} \|\xi\|^{\alpha_2}.\]
Otherwise, we have
\[ [f]_{C_\ell^{\alpha_3}} \|\xi\|^{\alpha_3} + \left( [f]_{C_\ell^{\alpha_1}}^{\bar \theta} [f]_{C_\ell^{\alpha_3}}^{1-\bar \theta} + [f]_{C_\ell^{\alpha_1}} \right)\|\xi\|^{\bar \alpha} \lesssim [f]_{C_\ell^{\alpha_1}}^{\theta} [f]_{C_\ell^{\alpha_3}}^{1-\theta} \|\xi\|^{\alpha_2} + [f]_{C_\ell^{\alpha_1}} \|\xi\|^{\bar \alpha} .\]

If $\alpha_2 > \bar \alpha$, we would have $q^2_z = q^3_z$ and the term $|q^3_z-q^1(z)| \leq \left( [f]_{C_\ell^{\alpha_1}}^{\bar \theta} [f]_{C_\ell^{\alpha_3}}^{1-\bar \theta} + [f]_{C_\ell^{\alpha_1}} \right)\|\xi\|^{\bar \alpha}$ would appear on the first line of the inequality. If $(\mathbb N + 2s \mathbb N) \cap [\alpha_1, \alpha_3) = \emptyset$, we have  $q_z^1 = q^2_z = q^3_z$ and the extra term involving $\bar \alpha$ would not be there.
\end{proof}

\subsection{Discussion about choices of distance}

\label{s:discussion}

We use the distance $d_\ell$ which is invariant by left translations
and by the scaling $S_R$. Let us analyze the consequences of this
choice and compare with the other possible choices.

We could define a distance $d_r$ that is invariant by right translations of the Lie group. It would be given by:
\[ d_r(z_1,z_2) := \max \left\{ |t_2-h| + |h-t_1|, |x_1-x_2+h(v_1-v_2)|^{2s/(1+2s)} , |v_1-v_2|^{2s} \right\}^{1/2s}. \]
Moreover, it would be comparable to the following expressions.
\[ d_r(z_1,z_2) \approx \begin{cases}
\| z_2 \circ z_1^{-1} \|, \\
\| z_1 \circ z_2^{-1} \|, \\
\inf_{h \in \R} \left\{ |t_2-h|^{1/(2s)} + |h-t_1|^{1/(2s)} + |x_1-x_2+h(v_1-v_2)|^{1/(1+2s)} + |v_1-v_2| \right\}.
\end{cases} \]

Alternatively, we could ignore the Lie group structure and define a distance $d_s$ that only takes scaling into account.
\[ d_s(z_1,z_2) := \|z_1 - z_2\|.\]
Here $\|.\|$ stands for the scaled norm as in \eqref{e:norm}.

The most brutal choice would be to ignore both the Lie group action and scaling and use the plain Euclidean distance in $\R^{1+2d}$.
\[ d_e(z_1,z_2) := \left( |t_1-t_2|^2 + |x_1-x_2|^2 + |v_1-v_2|^2 \right)^{1/2}.\]

The definition of H\"older spaces (Definition \ref{d:holder-space})
depends on the choice of distance. We can thus consider the four
possible candidates $C_\ell^\alpha$, $C_r^\alpha$, $C_s^\alpha$ and
$C_e^\alpha$. The distances are not equivalent, and these four spaces
are all different. Their only equivalence appears when measuring
distances from the origin
$d_\ell (0,z) \approx d_r(0,z) \approx d_e(0,z)$. Thus
$C^\alpha_\ell(0) = C^\alpha_r(0) = C^\alpha_e(0)$ (by $C^\alpha_\ast(0)$ we mean the functions that are $C^\alpha_\ast$ at the point $0$).

The class of equations \eqref{e:main} is invariant by left
translations. Because of that, the norm $d_\ell$ is the most
appropriate to work with. For example, if we proved an estimate for
solutions of \eqref{e:main} of the sort
$[f]_{C_\ell^\alpha(0)} \lesssim \|f\|_{L^\infty(Q_1)}$, it implies by
simple translations that
$[f]_{C_\ell^\alpha(Q_{1/2}) } \lesssim \|f\|_{L^\infty(Q_1)}$. This
implication does not hold true for $C_r^\alpha(Q_{1/2})$,
$C_s^\alpha(Q_{1/2})$ or $C_e^\alpha(Q_{1/2})$ (at least not true
keeping the same exponent $\alpha$).

In previous works, people have taken more or less attention to these
distinctions. The results in \cite{golse2016harnack} and
\cite{imbert2016weak} are oblivious of the choice of distance. That is
because these results are about an estimate in H\"older spaces for an
undetermined exponent $\alpha>0$. For any pair of points $z_1$ and
$z_2$ in $Q_1$, the following inequality holds
\[ d_1(z_1,z_2) \leq C d_2(z_1,z_2)^{1/(1+2d)},\]
where $d_1$ and $d_2$ are any two choices among $d_\ell$, $d_r$, $d_s$
and $d_e$. Thus, the main theorems in \cite{golse2016harnack} and
\cite{imbert2016weak} hold for the $C_\ell^\alpha$ norm defined in terms of
any of these distances, modulo an adjustments of the constants and
H\"older exponent $\alpha$.

For Schauder estimates, the distinction between different distances
plays a crucial role. In this case we want to obtain an estimate with
the precise exponent $C^{2s+\alpha}$ when the right hand side is
$C^\alpha$. It seems that such an estimate can only be true with the
distance $d_\ell$.

For right-invariant H\"older spaces $C^\alpha_r$ in terms of $d_r$,
the corresponding statement of Lemma \ref{l:holder-derivatives} would
be in terms of the operators $\partial_t$, $\partial_{x_i}$ and
$\partial_{v_i} + v_i \partial_{x_i}$. These differential operators
have the advantage that they commute with the equation
\eqref{e:kinetic-constant}. For regular H\"older spaces $C^\alpha_s$
or $C^\alpha_e$, Lemma \ref{l:holder-derivatives} would of course hold
with pure derivatives $\partial_t$, $\partial_{x_i}$ and
$\partial_{v_i}$.

Our Liouville theorem~\ref{t:liouville} holds for any choice of distance
$d_\ell$, $d_r$ or $d_s$. This is because in the step 1 of the proof
we establish that the function is constant in $x$. After that, we
ignore the $x$ coordinate and the three distances are the same.

In the proof of Lemma \ref{l:holder-exponent-improvement}, we select a
sequence of functions $\tilde f_j$ that are scaled left-translations
of a sequence of solutions $f_j$. If we used a different choice of
distance that is not invariant by left translations, we would not be
able to conclude anything about their $C^\beta$ norms.

One needs to be careful throughout this paper to make sure we do not
implicitly use the exact triangle inequality for $s < 1/2$, we do not commute group
operation $\circ$, and that we do not accidentally apply $d_r$ or
$d_s$ instead of $d_\ell$.

\section{Integral operators}
\label{s:integral}

This section is devoted to the integral coperators
\begin{equation} \label{e:L}
  \LL f(v) = \int_{\R^d} (f(v') - f(v)) K(v'-v) \dd v'  
\end{equation}
associated with fixed kernels $K$ from the
elliptic class $\K$ given in Definition~\ref{d:class-of-kernels}. 
We first explain when these integral operators can be
evaluated pointwise. We then turn to limits of kernels and integral
operators.  We conclude this section by proving H\"older estimates
that will be used in the proof of the Schauder estimate.

\subsection{Evaluating operators pointwise}

In this subsection, we discuss how to evaluate pointwise operators associated with kernels in the elliptic class $\K$. 
 More precisely, we want to explain the conditions
that a function $f: \R^d \to \R$ must meet in order for the integral
in \eqref{e:L} to be well defined at the point $v_0$. On one hand, it
must be sufficiently regular so that the integral does not diverge in
a neighborhood of $v_0$. On the other hand, it must also satisfy some
growth conditions so that the integral does not diverge at
infinity. Let us split the domain of integration accordingly and
analyze conditions for convergence of each part.

\[ \LL f(v_0) = \PV \int_{B_1} (f(v_0+w) - f(v_0)) K(w) \dd w + \int_{\R^d \setminus B_1} (f(v_0+w) - f(v_0)) K(w) \dd w.\]

When $s \geq 1/2$, the first term must be understood in the principal
value sense, even when $f$ is smooth. Using the symmetry condition
$K(w) = K(-w)$, we can symmetrize the integral and remove the
principal value.
\[ \PV \int_{B_1} (f(v_0+w) - f(v_0)) K(w) \dd w = \frac 12 \int_{B_1} (f(v_0+w) + f(v_0 - w) - 2f(v_0)) K(w) \dd w.\]
Because of \eqref{e:K-upper-bound}, this integral is classically computable when $f \in C^{2s+\eps}(B_1(v_0))$ for some $\eps>0$. Indeed,
\begin{align} 
\nonumber  \frac 12 \int_{B_1} |f(v_0+w) + f(v_0 - w) - 2f(v_0)| K(w) \dd w &\lesssim [f]_{C^{2s+\eps}(B_1(v_0))} \left( \int_{B_1} |w|^{2s+\eps} K(w) \dd w \right), \\
\label{e:d2v0}                                                                   &\lesssim \Lambda [f]_{C^{2s+\eps}(B_1 (v_0))}.
\end{align}

In order to analyze the tail of the integral, we introduce the following function
\begin{equation} \label{e:majorant}
 \omega_{v_0}(r) := \sup \{ |f(v)| : v \in B_{2r}(v_0) \setminus B_{r/2}(v_0) \}.
\end{equation}
We observe that, because of \eqref{e:K-upper-bound}, the function $\omega_{v_0}$ can be used to bound the tail of the integral.
We state the estimate in a lemma for later use. 
\begin{lemma} \label{l:tail}
Let $R>0$. 
  \[ \int_{\R^d \setminus B_R} f(v_0+w)  K(w) \dd w \lesssim \Lambda \int_{R/2}^\infty \omega_{v_0}(r) r^{-1-2s} \dd r.\]
\end{lemma}
\begin{proof}
Using the definition of $\omega_{v_0} (r)$, we can write
  \begin{align*}
    \int_{\R^d \setminus B_R} f(v_0+w)  K(w) \dd w & \le    \int_{\R^d \setminus B_R} \min \{  \omega_{v_0} (r): r \in [|w|/2,2 |w|]\}  K(w) \dd w  \\
                                                   & \le    \int_{\R^d \setminus B_R} \left\{ \frac{2}{3|w|} \int_{|w|/2}^{2|w|}  \omega_{v_0} (r) \dd  r  \right\}  K(w) \dd w  \\
                                                   & \lesssim    \int_{\R^d \setminus B_R} \left\{ \int_{|w|/2}^{2|w|} \frac{ \omega_{v_0} (r)}r \dd  r  \right\}  K(w) \dd w  \\
    & \lesssim \int_{R/2}^{\infty} \left\{ \int_{B_{2r}\setminus B_{r/2}} K (w) \dd w \right\} \frac{ \omega_{v_0} (r)}r  \dd r.
    \end{align*}
Using \eqref{e:K-upper-bound} yields the result. 
\end{proof}

Summarizing, we have the estimate
\begin{equation} \label{e:L-computable}
 \LL f(v_0) \lesssim \Lambda \left( [f]_{C^{2s+\eps}(v_0)} + \int_1^\infty (|f(v_0)| + \omega_{v_0}(r)) r^{-1-2s} \dd t \right).
\end{equation}
Moreover, the integral expression in \eqref{e:L} is classically computable whenever the right hand side of the inequality is finite.

\subsection{Weak limits of kernels}

We now discuss how to pass to the limit in kernels. We first define
the notion of weak-$\ast$ convergence and we then prove that the set
$\K$ is compact for the corresponding topology.

\begin{defn}[Weak-$\ast$ convergence of
  kernels] \label{d:kernel-weak-limit} We say that a sequence $K_j$ of
  Radon measures in $\R^d \setminus \{0\}$ converges weakly-$\ast$ to
  the Radon measure $K_\infty$ if for any continuous function
  $\varphi : \R^d \to \R$, compactly supported, whose support does not
  include the origin, we have
\[ \lim_{j \to \infty} \int_{\R^d} \varphi(w) K_j(w) \dd w = \int_{\R^d} \varphi(w) K_\infty(w) \dd w.\]
\end{defn}
\begin{lemma}[Closedness of $\K$ under weak-$\ast$
  limit] \label{l:weak-limit-stay-in-class} If the kernels $K_j$
  belong to the class $\K$ of Definition \ref{d:class-of-kernels} and
  $K_j$ converges weakly-$\ast$ to $K_\infty$, then $K_\infty$ also
  belongs to the class $\K$.
\end{lemma}
\begin{proof}
  The fact that $K_\infty$ is a non-negative Radon measure is
  classical.

  As far as the upper bound is concerned, it is enough to consider a
  cut-off function $\varphi_r$ valued in $[0,1]$ with $\varphi_r \equiv 1$
  in $B_r \setminus B_\eta$ and whose compact support is contained in
  $B_{r+\eps} \setminus B_{\eta/2}$ for some $\eps, \eta >0$. Then we
  write
  \[ \int_{\R^d} |w|^2 \varphi_r (w) K_j (w) \dd w \le \Lambda (r+\eps)^{2-2s}. \]
  Passing to the limit as $j\to +\infty$, we get
  \[ \int_{B_r \setminus B_\eta}
    |w|^2  K_\infty (w) \dd w \le
    \int_{\R^d} |w|^2 \varphi_r (w) K_j (w) \dd w \le \Lambda (r+\eps)^{2-2s}. \]
  Since $\eps$ and $\eta$ are arbitrary, $K_\infty$ satisfies the upper bound.

  As far as the coercivity estimate is concerned, let $R>0$ and
  $\varphi \in C^2(B_R)$. Since $K_j \in \K$, we have
  \[
    \iint_{B_R \times B_R} |\varphi(v) - \varphi(v')|^2 K_j(v'-v) \dd v' \dd v
    \geq \lambda \iint_{B_{R/2} \times B_{R/2}} |\varphi(v) - \varphi(v')|^2 |v'-v|^{-d-2s} \dd v' \dd v.
\]

For all $r >0$, consider a cut-off function $\Psi_r$ valued in
$[0,1]$, $\Psi_r \equiv 1$ in $B_{r/2}$ and $\Psi_r \equiv 0$ outside
$B_r$.  Thanks to the uniform upper bound, we
have
\[
  \iint_{B_R \times B_R} \Psi_r (v'-v) |\varphi(v) - \varphi(v')|^2 K_j(v'-v) \dd v' \dd v
  \lesssim \Lambda R^d \|\nabla \varphi\|_{L^\infty (B_R)} r^{2-2s}.
\]
Combining the two previous estimates, we get
\begin{multline*}
  \iint_{B_R \times B_R} (1-\Psi_r) (v'-v) |\varphi(v) - \varphi(v')|^2 K_j(v'-v) \dd v' \dd v \\
  \ge \lambda \iint_{B_{R/2} \times B_{R/2}} |\varphi(v) - \varphi(v')|^2 |v'-v|^{-d-2s} \dd v' \dd v - \mathcal{O} (r^{2-2s}).
\end{multline*}
We can now pass  to the limit as $j \to \infty$ and obtain
\begin{multline*}
  \iint_{B_{R} \times B_{R}} |\varphi(v) - \varphi(v')|^2 K_\infty(v'-v) \mathbf{1}_{|v'-v| \ge r/2} \dd v' \dd v \\
  \ge \lambda \iint_{B_{R/2} \times B_{R/2}} |\varphi(v) - \varphi(v')|^2 |v'-v|^{-d-2s} \dd v' \dd v - \mathcal{O} (r^{2-2s}).
\end{multline*}
Letting $r \to 0^+$  yields the result. 
\end{proof}
\begin{lemma}[Compactness of $\K$ for weak-$\ast$
  topology] \label{l:weak-limit-compactness} If $K_j$ is a sequence of
  kernels in the class $\K$ of Definition \ref{d:class-of-kernels},
  then it has a weak-$\ast$ convergent subsequence.
\end{lemma}
\begin{proof}
  We split $\R^d \setminus \{0\}$ into $\bigcup_{k \in \mathbb Z} C_k$
  with $C_k = B_{2^{k-1}} \setminus B_{2^{k-1}}$. The sequence of Radon
  measures $\{ K_j |_{C_k} \}_j$ in $C_k$ are compact because of
  Banach-Alaoglu theorem. Thanks to a diagonal argument, we can thus
  extract a sequence $\{K_{l(j)}\}_j$ converging towards $K_\infty$ on
  each ring $C_k$. In particular, this sequence weak-$\ast$ converges
  to $K_\infty$ in the sense of Definition~\ref{d:kernel-weak-limit}.
\end{proof}

\subsection{Limits of operators}

\begin{lemma}[Limits of operators] \label{l:weak-limit-stability} Let
  $\Omega$ be a open bounded set of $(-\infty,0]\times \R^d \times \R^d$
  and $K_j$ and $f_j:(-\infty,0]\times \R^d \times \R^d \to \R$ be a
  sequence of kernels and functions respectively so that the following
  conditions hold.
\begin{enumerate}
\item\label{1} Each $K_j$ belongs to the class $\K$.
\item\label{2} The sequence $K_j \rightharpoonup K_\infty$ weakly-$\ast$ as $j \to \infty$.
\item\label{3} The sequence $f_j \to f$ locally uniformly in $(-\infty,0]\times \R^d \times \R^d$ as $j \to \infty$.
\item\label{4} The sequence $f_j$ is uniformly bounded in $C_\ell^{2s+\eta}(\Omega)$ for some $\eta > 0$.
\item\label{5} There is a function $\omega: [1,\infty) \to \R$ so that
  \( \int_{1}^\infty \omega(r) r^{-1-2s} \dd r < +\infty\) and for every
  $j \in \mathbb N$,
  \[  \forall r \ge 1, \; \forall  v \in B_r \setminus B_{r/2}, \qquad  f_j(v) \le \omega(r).\]
\end{enumerate}
 Then we have
\[ \LL_j f_j \to \LL_\infty f \qquad \text{locally uniformly in } \Omega \text{ as } j \to \infty\]
where $\LL_j$ is the integral operator corresponding to $K_j$, see \eqref{e:L}.
\end{lemma}
\begin{proof}
  Let $\eps>0$ be arbitrary.  

We use the assumption \eqref{4} to bound the part of the integrals in $\LL_j f_j$ and $\LL f$ around the origin. Thanks to the symmetry assumption of the kernels,
\begin{equation} \label{e:wc1}
 \begin{aligned}
\left \vert \PV \int_{B_\rho} (f_j(v+w) - f_j(v)) K_j(w) \dd w \right \vert &= \left \vert \frac 12 \int_{B_\rho} (f_j(v+w) + f_j(v-w) - 2f_j(v)) K_j(w) \dd w \right \vert, \\
&\leq \frac 12 [f_j]_{C_\ell^{2s+\eta}(\Omega)} \int_{B_\rho} |w|^{2s+\eta} K_j(w) \dd w, \\
&\lesssim \Lambda \rho^{\eta} < \eps/8\qquad  \text{provided that $\rho$ is sufficiently small.}
\end{aligned} 
\end{equation}

We use the assumption \eqref{5} to bound the tails of the integrals. Note that for any $v \in \Omega$ and $r > \diam(\Omega)$, we can obtain a common majorant function $\omega_v(r)$ for all functions $f_j$, as in \eqref{e:majorant}, by the formula
\[ \omega_v(r) \leq \omega(r - \diam(\Omega)) + \|f_j\|_{L^\infty}.\]
Using Lemma \ref{l:tail}, for $R$ sufficiently large,
\begin{equation} \label{e:wc2}
\left \vert \int_{\R^d \setminus B_R} (f_j(v+w) - f_j(v)) K_j(w) \dd w \right \vert \leq \int_{R/2}^\infty  \left( \omega(r - \diam(\Omega)) + \|f_j\|_{L^\infty} \right) r^{-1-2s} \dd r < \eps/8
\end{equation}

Using that $K_j \in \K$ and $f_j \to f$ locally uniformly, then for $j$ sufficiently large
\begin{equation} \label{e:wc3}
 \int_{B^R \setminus B_\rho} |f_j(v+w) - f_j(v) - f(v+w) + f(v)| K_j(w) \dd w \lesssim \Lambda r^{-2s} \|f_j - f\|_{L^\infty(B_R + \Omega)} < \eps/4.
\end{equation}

Finally, since $K_j \to K$ weak-$\ast$, then for $j$ large
\begin{equation} \label{e:wc4}
 \left\vert \int_{B_R \setminus B_\rho} (f(v+w) - f(v)) (K_j-K) \dd w \right \vert < \eps/4.
\end{equation}
Note that because $f$ is a continuous function on $\R^{1+2d}$, the choice of $j$ can be made uniform with respect to the point $z = (t,x,v) \in \Omega$ by the argument that led to \eqref{e:wc3}.

Adding up \eqref{e:wc1}, \eqref{e:wc2}, \eqref{e:wc3} and \eqref{e:wc4}, we get that $|\LL_j f - \LL f| < \eps$ uniformly in $\Omega$ for $j$ sufficiently large.

\end{proof}

\subsection{Consequences of Assumption~\ref{a:K-holder}}

We gather here some consequences of Assumption~\ref{a:K-holder} that
will be used in the next subsection when deriving H\"older estimates.
\begin{align}
  \label{e:a0-2s+alpha}
  \int_{|w| \le 1} |w|^{2s+\alpha} |K_{z_1} (w) - K_{z_2} (w)| \dd w \lesssim A_0 d_\ell(z_1,z_2)^\alpha,\\
  \label{e:a0-infty}
  \int_{|w| \ge 1} |K_{z_1} (w) - K_{z_2} (w)| \dd w \lesssim A_0 d_\ell(z_1,z_2)^\alpha.
\end{align}
Both inequalities are consequences of the fact that \ref{a:K-holder} implies that for all $r>0$, 
\[ \int_{B_r \setminus B_{r/2}} |K_{z_1} (w) - K_{z_2} (w)| \dd w \lesssim A_0 r^{-2s} d_\ell (z_1,z_2)^\alpha.\]
To get \eqref{e:a0-infty}, we use dyadic rings $B_{2^{k+1}} \setminus B_{2^k}$ and sum over $k$.

\subsection{H\"older  estimates}

We gather here estimates that will be used when proving the main
Schauder estimate, see the terms $A$ and $B$ on page~\pageref{p:ab}.

Let us consider a sign changing kernel $K$ such that $K(w)=K(-w)$ and it satisfies the upper bound
for all $r>0$, 
\begin{equation} \label{e:sign-changing-upper-bound}
 \int_{B_r} |w|^2 |K(w)| \dd w \leq \Lambda r^{2-2s}.
\end{equation}
Let us study the corresponding integral operator 
\[ \LL_K f(z) = \int (f(z \circ (0,0,w)) -f(z)) K(w) \dd w .\]
We start with a global estimate.
\begin{lemma}\label{l:estimA0} Assume $\alpha < \min(1,2s)$. For any sign-changing symmetric kernel $K$ satisfying \eqref{e:sign-changing-upper-bound}, and  {$f \in C_\ell^{2s+\alpha}(\R^{1+2d})$}, we have the estimate
\[ [\LL_K f]_{C_\ell^{\alpha}(\R^{2d+1})} \lesssim \Lambda [f]_{C_\ell^{2s+\alpha} (\R^{2d+1})}.\]
\end{lemma}

\begin{proof}
Let us start by fixing some notation. As usual, we denote by $p_z$ the polynomial expansion of $f$ at $z$ so that $\deg_k p_z < 2s+\alpha$ and for $z, \xi \in \R^{1+2d}$,
\[ |f(z \circ \xi) - p_z(\xi)| \leq [f]_{C_\ell^{2s+\alpha}} \|\xi\|^{2s+\alpha}.\]

Let us also write, for $z,\xi, \zeta \in \R^{1+2d}$,
\[ \Delta_\xi f(z) = f(z \circ \xi) - f(z) \qquad \text{ and } \qquad \delta_\zeta p_z(\xi) = p_{z \circ \zeta}(\xi) - p_z(\xi).\]

We must estimate the following quantity
\begin{align*} 
 \Delta_\xi \LL_K f(z) &= \LL_K f(z \circ \xi) - \LL_K f(z), \\
&= \int_{\R^d} [ f(z \circ \xi \circ (0,0,w)) - f(z \circ (0,0,w)) - f(z \circ \xi) + f(z)] K(w) \dd w, \\
&= \int_{\R^d} [ \Delta_\xi \Delta_{(0,0,w)} f(z)] K(w) \dd w.
\end{align*}

Since $\alpha< \min(1,2s)$, proving en estimate for $[\LL_K f]_{C_\ell^\alpha}$ amount to finding the right upper bound for $| \Delta_\xi \LL_K f(z)|$.

We split the integral above into two subdomains: $B_R$ and $\R^d \setminus B_R$. We will later choose $R = \|\xi\|$.

Estimating the integral in $B_R$, we symmetrize using that $K(w) = K(-w)$ and 
\begin{align*}
\left \vert \int_{B_R} \Delta_{(0,0,w)}f(z) K(w) \dd w \right \vert &= \left \vert \int_{B_R} [ f(z \circ (0,0,w)) - f(z) ] K(w) \dd w \right \vert , \\
&= \frac 12 \left \vert \int_{B_R} [ f(z \circ (0,0,w))+f(z \circ (0,0,-w))  - 2f(z) ] K(w) \dd w \right \vert, \\
\intertext{Here we use that $|f(z \circ (0,0,w)) - p_z(0,0,w)| \leq [f]_{C_\ell^{2s+\alpha}} |w|^{2s+\alpha}$. The polynomial $p_z$ has kinetic degree smaller than $2s+\alpha$. The first order terms in $v$ cancel out by the symmetrization. There may be second order terms in $v$ if $2s+\alpha > 2$. There cannot be higher order terms in $v$ with our restrictions on $s$ and $\alpha$. Any term involving $t$ and $x$ vanishes when evaluating on $(0,0,w)$. Thus, when $2s+\alpha$ we continue using the assumption \eqref{e:sign-changing-upper-bound}}
&\leq \frac 12 [f]_{C_\ell^{2s+\alpha}} \int_{B_R} |w|^{2s+\alpha} |K(w)| \dd w, \\
&\lesssim [f]_{C_\ell^{2s+\alpha}} \Lambda R^{\alpha} \\
\end{align*}

When $2s+\alpha>2$, we cannot cancel out the second order terms in $v$ in the polynomial $p_z$. Thus, in that case the same computation leads to
\[
 \left \vert \int_{B_R} \Delta_{(0,0,w)} f(z) K(w) \dd w - \partial_{v_i v_j}f(z) \left( \int_{B_R} w_i w_j K(w) \dd w \right) \right \vert \lesssim [f]_{C_\ell^{2s+\alpha}} \Lambda R^{\alpha}.
\]

In the estimates above, the value of $z \in \R^{1+2d}$ is arbitrary. The inequalities hold for $z \circ \xi$ just as well. Therefore, applying $\Delta_\xi$ we get
\begin{align*} 
\left \vert  \int_{B_R}  \Delta_\xi \Delta_{(0,0,w)} f(z) K(w) \dd w \right \vert &\lesssim  [f]_{C_\ell^{2s+\alpha}} \Lambda R^{\alpha} \\
& \qquad \text{ ( if $2s+\alpha>2$ ) }+ |\Delta_\xi D^2 f| \Lambda R^{2-2s}, \\
 &\lesssim  [f]_{C_\ell^{2s+\alpha}} \Lambda R^{\alpha} \\
& \qquad \text{ ( if $2s+\alpha>2$ ) }+ \| \xi \|^{2s+\alpha-2} \Lambda R^{2-2s}, 
\end{align*}
In the last inequality we used Lemma \ref{l:holder-derivatives} for the case $2s+\alpha>2$. Note that for any value of $2s+\alpha$, when we choose $R = \|\xi\|$ we will get
\[ \left \vert  \int_{B_R}  \Delta_\xi \Delta_{(0,0,w)} f(z) K(w) \dd w \right \vert \lesssim [f]_{C_\ell^{2s+\alpha}} \Lambda \|\xi\|^{\alpha}.\]

Now we move on to estimate the part of the integral in $\R^d \setminus B_R$. We use the following two inequalities
\begin{align}
| f(z \circ \xi) - p_{z} ( \xi ) | &\leq [f]_{C_\ell^{2s+\alpha}} \| \xi \|^{2s+\alpha},  \label{e:f1} \\
| f(z \circ \xi \circ (0,0,w)) - p_{z \circ (0,0,w)} ((0,0,w)^{-1} \circ \xi \circ (0,0,w)) | &\leq [f]_{C_\ell^{2s+\alpha}} \| (0,0,w)^{-1} \circ \xi \circ (0,0,w) \|^{2s+\alpha}. \label{e:f2}
\end{align}

The second one naturally requires some further analysis. We observe that
\[ (0,0,w)^{-1} \circ \xi \circ (0,0,w) = (\xi_1, \xi_2 + \xi_1 w, \xi_3).\]
Therefore
\begin{equation} \label{e:f3}
 \| (0,0,w)^{-1} \circ \xi \circ (0,0,w)  \| \lesssim \|\xi\| + ( |\xi_1| |w| )^{1/(1+2s)} \leq \|\xi\| + \|\xi_1\|^{\frac{2s}{1+2s}}  |w|^{\frac 1 {1+2s}} .
\end{equation}

We will split the integral in $\R^d \setminus B_R$ as the sum of several terms.
\[ \left \vert \int_{\R^d} [ \Delta_\xi \Delta_{(0,0,w)} f(z)] K(w) \dd w \right\vert \leq I_1 + I_2 + I_3 + I_4,\]
where
\begin{align*} 
 I_1 &:= \int_{\R^d \setminus B_R} | f(z \circ \xi) - p_{z} ( \xi ) | K(w) \dd w, \\
I_2 &:=  \int_{\R^d \setminus B_R} | f(z \circ \xi \circ (0,0,w)) - p_{z \circ (0,0,w)} ((0,0,w)^{-1} \circ \xi \circ (0,0,w)) | K(w) \dd w, \\
I_3 &:=  \int_{\R^d \setminus B_R} |  \delta_{(0,0,w)} p_{z} ( \xi ) - \Delta_{(0,0,w)} f(z) | K(w) \dd w, \\
I_4 &:=  \int_{\R^d \setminus B_R} |  p_{z \circ (0,0,w)} ((0,0,w)^{-1} \circ \xi \circ (0,0,w)) - p_{z \circ (0,0,w)}(\xi) | K(w) \dd w.
\end{align*}

We bound $I_1$ easily using \eqref{e:f1} and \eqref{e:sign-changing-upper-bound}.
\begin{equation} \label{e:I1}
 I_1 \lesssim [f]_{C_\ell^{2s+\alpha}} \Lambda \|\xi\|^{2s+\alpha} R^{-\alpha}.
\end{equation}

We bound $I_2$ following the procedure, but applying \eqref{e:f3}.
\begin{equation} \label{e:I2}
\begin{aligned}
 I_2 &\lesssim [f]_{C_\ell^{2s+\alpha}} \left( \Lambda \|\xi\|^{2s+\alpha} R^{-\alpha} + \|\xi\|^{\frac{2s (2s+\alpha)}{1+2s}}  \int_{\R^d \setminus B_R} |w|^{\frac{2s+\alpha}{1+2s}} K(w) \dd w \right), \\
&\lesssim [f]_{C_\ell^{2s+\alpha}} \Lambda \left(  \|\xi\|^{2s+\alpha} R^{-\alpha} + \|\xi\|^{\frac{2s (2s+\alpha)}{1+2s}} R^{\frac{2s+\alpha}{1+2s}  -2s} \right).
\end{aligned} 
\end{equation}

For the analysis of $I_3$, we write $p_z$ as a sum of monomials.
\[ p_z(\xi) = \sum_{\deg_k m_j < 2s+\alpha} a_j(z) m_j(\xi).\]
Moreover $a_0(z) = f(z)$ and
\[ \delta_{(0,0,w)} p_z(\xi) = \sum_{\deg_k m_j < 2s+\alpha} \Delta_{(0,0,w)} a_j(z) m_j(\xi).\]
From Lemma \ref{l:holder-derivatives}, we know that $[a_j]_{C_\ell^{2s+\alpha-\deg_k m_j}} \leq [f]_{C_\ell^{2s+\alpha}}$. Note that $2s+\alpha-\deg_k m_j < 1$ for any monomial such that $0 < \deg_k m_j < 2s+\alpha$. Thus,
\[ |  \delta_{(0,0,w)} p_{z} ( \xi ) - \Delta_{(0,0,w)} f(z) |  \lesssim [f]_{C_\ell^{2s+\alpha}} \sum_{0 < \deg_k m_j < 2s+\alpha} |w|^{2s+\alpha-\deg_k m_j} \|\xi\|^{\deg_k m_j}.\]
Therefore,
\[ I_3 \lesssim [f]_{C_\ell^{2s+\alpha}} \Lambda \sum_{0 < \deg_k m_j < 2s+\alpha} R^{\alpha - \deg_k m_j} \|\xi\|^{\deg_k m_j}.\]

Regarding $I_4$, note that since $2s+\alpha < 1+2s$, the polynomial $p_z$ cannot have a term that involves its second component ($x$). Since $(0,0,w)^{-1} \circ \xi \circ (0,0,w)$ and $\xi$ differ only on their second component, then actually $I_4 = 0$.

When we choose $R = \|\xi\|$, the estimates of all terms are $\lesssim \Lambda  [f]_{C_\ell^{2s+\alpha}} \|\xi\|^{\alpha}$. And therefore we conclude the proof.
\end{proof}


We next derive a local estimate from the global one. 
\begin{lemma}\label{l:estimA0-loc}
  Let $\alpha = \frac{2s}{1+2s} \gamma$ with $\gamma < \min (1,2s)$. Then
  \( [\LL_K f]_{C_\ell^{\alpha}(Q_{1/2})} \lesssim [f]_{C_\ell^{2s+\alpha}(Q_1)}  + [f]_{C_\ell^\gamma((-1,0]\times B_1 \times \R^d)}.\)
\end{lemma}
\begin{proof}
  It is enough to write $\LL_K f(z) = \LL_{\tilde{K}}f (z) + C(z)$ where
  $\tilde{K} (w) = \mathbf{1}_{B_\rho} (w) K(w)$ and $C(z)$ is $\LL_K$
  where $K$ is replaced with  $(1-\mathbf{1}_{B_\rho}) K(w)$ and $\rho$
  small. From the previous lemma, we have
  \[
    [\LL_{\tilde K}f]_{C_\ell^{\alpha}(Q_{1/2})} \lesssim [f]_{C_\ell^{2s+\alpha}  (Q_1)}.
  \]
  Let us prove that
  \begin{equation}\label{e:C}
    [C]_{C_\ell^\alpha (Q_{1/2})} \lesssim [f]_{C_\ell^\gamma ((-1,0]\times B_1 \times \R^d)} .
  \end{equation}
  In order to do so, we write for $z_1,z_2 \in Q_{1/2}$, 
 \[
   C(z_2) -C(z_1) = \int_{\R^d \setminus B_\rho} (f(z_2 \circ (0,0,w))-f(z_2) -f (z_1 \circ (0,0,w)) + f(z_1 )) k(w) \dd w
 \]
 and we first prove that
 \begin{equation}\label{e:l4}
   | f(z_2 \circ (0,0,w))-f(z_2) -f (z_1 \circ (0,0,w)) + f(z_1 )| \lesssim [f]_{C_\ell^\gamma} (1+ |w|^{\frac\gamma{1+2s}}) d_\ell(z_1,z_2)^\alpha.
 \end{equation}
 On the one hand,  since $f \in C_\ell^\alpha (Q_{1/2})$ and $\alpha \le \gamma$, we have
 \[ | f(z_2) -f (z_1)| \lesssim [f]_{C_\ell^\gamma} d_\ell(z_1,z_2)^\alpha. \]
On the other hand, the $C_\ell^\gamma$ regularity of $f$ yields
 \begin{align*}
   | f(z_2 \circ (0,0,w)) -f (z_1 \circ (0,0,w))| & \leq [f]_{C_\ell^\gamma} d_\ell(z_1 \circ (0,0,w),z_2 \circ (0,0,w))^\gamma \\
                                                  & \lesssim [f]_{C_\ell^\gamma} \|  (0,0,w)^{-1} \circ z_2^{-1} \circ z_1 \circ (0,0,w)\|^\gamma .
 \end{align*}
 We now compute $(0,0,-w) \circ z_2^{-1} \circ z_1 \circ (0,0,w) = z_2^{-1} \circ z_1  - (0,(t_1-t_2)w,0)$, and get
 \begin{align*}
   \|  (0,0,w)^{-1} \circ z_2^{-1} \circ z_1 \circ (0,0,w)\| &\lesssim d_\ell (z_1,z_2) + |t_1-t_2|^{\frac1{1+2s}} |w|^{\frac1{1+2s}} \\
                                                             & \lesssim (1+ |w|^{\frac1{1+2s}}) d_\ell (z_1,z_2)^{\frac{2s}{1+2s}}.
 \end{align*}
 Combining the three previous estimates yields \eqref{e:l4}.  Since
 $\int_{\R^d \setminus B_\rho} (1+ |w|^{\frac\gamma{1+2s}}) k (w) \dd
 w \lesssim \Lambda C_\rho$, thanks to
 Assumption~\eqref{e:K-upper-bound} and the fact that $\gamma < 2s$,
 \eqref{e:l4} implies \eqref{e:C}. This achieves the proof of the
 lemma.
\end{proof}

\subsection{The local H\"older estimate}

The symmetry condition $K(t,x,v,v+w) = K(t,x,v,v-w)$ corresponds to
equations in \emph{non-divergence form}, in the sense that the
integro-differential operator has a structure similar to that of
elliptic equations of non-divergence form (as in
$a_{ij}(t,x,v) \partial_{v_i v_j} f$). It is different of the other
symmetry condition that would make the operator self adjoint
$K(t,x,v,v') = K(t,x,v',v)$, and corresponds to equations in
\emph{divergence form}. The weak Harnack inequality, in the style of
De Giorgi, obtained in \cite{imbert2016weak} does not apply to
\eqref{e:main} precisely because of this distinction of symmetry
assumptions. Our kernels $K$ do not satisfy the cancellation
conditions (1.6) and (1.7) from \cite{imbert2016weak}.

The situation is simpler when we take a translation invariant kernel
$K(w)$ and consider the equation
\begin{equation} \label{e:kinetic-constant}
 f_t + v \cdot \nabla_x f = \LL f + c(t,x,v).
\end{equation}

It is an integro-differential analog of an equation with constant
coefficients. There is no distinction in this case between divergence
and non-divergence form. The kernel $\tilde K(t,x,v,v') = K(v'-v)$
satisfies the symmetry condition (and thus also the cancellation
condition) $\tilde K(t,x,v,v') = \tilde K(t,x,v',v)$ for any kernel
$K$ in $\K$. 

The regularity of the solution $f$ to \eqref{e:kinetic-constant} is
not important. It is straight forward to approximate any
(weak/viscosity) solution to \eqref{e:kinetic-constant} with
$C^\infty$ solutions by mollification. Indeed, if $f$ solves
\eqref{e:kinetic-constant}, then for any smooth compactly supported
function $\varphi : \R^{1+2d} \to \R$, the function
\[
  \varphi \ast_k f (z) = \int_{\R^{1+2d}} \varphi(\xi) f(\xi \circ z)  \dd \xi,
\]
also solves \eqref{e:kinetic-constant} (perhaps in a slightly smaller
domain depending on the support of $\varphi$). Naturally, the function
$\varphi \ast_k f \in C^\infty$ whenever $\varphi \in
C^\infty$. Taking $\varphi$ to be an approximation of the unit mass at
the origin, we approximate any solution of \eqref{e:kinetic-constant}
by a smooth one. Therefore, we can safely assume, without loss of
generality, that every function is $C^\infty$ for the purposes of the
results in this section.

We apply the main result from \cite{imbert2016weak} to our setting.
\begin{thm}[Local H\"older estimate] \label{t:weak-harnack} Let $\LL$
  be an integral operator corresponding to a kernel in the class $\K$
  (as in Definition~\ref{d:class-of-kernels}).
  $f: (-1,0] \times B_1 \times \R^d \to \R$ be a function that solves
  the equation \eqref{e:kinetic-constant} in $Q_1$.  Then, the
  following estimate holds
  \[ [f]_{C_\ell^\delta(Q_{1/2})} \leq C \left( \|f\|_{L^\infty((-1,0]
        \times B_1 \times \R^d)} + \|c\|_{L^\infty(Q_1)} \right).\]
  Here $\delta>0$ and $C$ are constants depending only on dimension
  and the parameters $\lambda$ and $\Lambda$ of Definition
  \ref{d:class-of-kernels}.
\end{thm}
\begin{remark}
  Note that Theorem~\ref{t:weak-harnack} and its corollaries below
  hold for several different choices of the distance function. See
  Section \ref{s:discussion}.
\end{remark}
\begin{proof}
  The coercivity conditions in the definition of the class
  $\mathcal{K}$ of elliptic kernels slightly differ from the
  coercivity condition imposed in \cite{imbert2016weak}. Let
  $f: \R^d \to \R$ be supported in some ball $B_{\bar R}$.
  \begin{align*}
    - \int f L_v f \dd v  & = \frac12 \iint (f'-f)^2 K (v,v') \dd v'\\
                          & \ge \frac12 \iint_{B_{4 \bar R} \times B_{4 \bar R}}  (f'-f)^2 K (v,v') \dd v'\\
                          & \ge \frac\lambda2  \iint_{B_{2 \bar R} \times B_{2 \bar R}}  (f'-f)^2 |v-v'|^{-d-2s} \dd v'\\
                          & \ge \frac\lambda2  \iint  (f'-f)^2 |v-v'|^{-d-2s} \dd v'
                            - \int_{B_{\bar R}} f^2 (v) \left\{ \int_{\R^d \setminus B_{2 \bar R}} K (v,v') \dd v' \right \} \dd v \\
        & \ge \frac\lambda2  \iint  (f'-f)^2 |v-v'|^{-d-2s} \dd v' - \bar \Lambda \int f^2 (v)  \dd v
  \end{align*}
  with $\bar \Lambda \simeq \bar R^{-2s}$. 
  The other conditions from \cite{imbert2016weak} are satisfied straightforwardly from our assumptions in Definition \ref{d:class-of-kernels}.
\end{proof}

Note that the right hand side depends on the $L^\infty$ norm of $f$
with respect to all values of $v \in \R^d$. This is a common
inconvenience with nonlocal equations. The result can be easily
improved to allow functions $f$ that are unbounded as
$|v| \to \infty$. Let $\omega_0$ be the majorant function as in
\eqref{e:majorant}, centered at the origin. That is
\[ \omega_0(t,x,r) = \sup_{v \in B_{2r} \setminus B_{r/2}} f(t,x,v).\]
We derive the following improvement of Theorem \ref{t:weak-harnack}.
\begin{cor}
  Let $\LL$ be an integral operator corresponding to a kernel in the
  class $\K$ (as in Definition~\ref{d:class-of-kernels}).
  $f: (-1,0] \times B_1 \times \R^d \to \R$ be a function that solves
  the equation \eqref{e:kinetic-constant} in $Q_1$.  Then, the
  following estimate holds
  \[ [f]_{C_\ell^\delta(Q_{1/2})} \leq C \left( \|f\|_{L^\infty(Q_1)} +
      \|c\|_{L^\infty(Q_1)} + \sup_{t \in (-1,0], x \in B_1}
      \int_{1/2}^\infty \omega_0(t,x,r) r^{-1-2s} \dd r \right).\] Here
  $\delta>0$ and $C$ are constants depending only on dimension and the
  parameters $\lambda$ and $\Lambda$ of Definition
  \ref{d:class-of-kernels}.
\end{cor}
\begin{proof}
  We consider a $C^\infty$ function $\eta : \R^d \to [0,1]$ so that
  $\eta(v) = 1$ if $v \in B_{3/2}$ and $\eta(v)=0$ when
  $v \notin B_2$. We apply Theorem \ref{t:weak-harnack} to the
  localized function $\tilde f(t,x,v) = f(t,x,v) \eta(v)$. We must
  analyze the equation that $\tilde f$ satisfies. We compute directly
  $\LL \tilde f$ to get
\[ \tilde f_t + v \cdot \nabla_x \tilde f - \LL \tilde f = \tilde{c} - h(t,x,v) \qquad \text{in } Q_1,\]
where $\tilde{c} = \eta c$ and 
\[ h(t,x,v) = \int_{\R^d \setminus B_{3/2}} f(t,x,v') (\eta(v') - 1) K(t,x,v,v') \dd v'.\]
From Theorem \ref{t:weak-harnack}, we have
\begin{equation} \label{e:wh1}
 \|f\|_{C_\ell^\delta(Q_{1/2})} = \|\tilde f\|_{C_\ell^\delta(Q_{1/2})} \leq C \left( \|\tilde f\|_{L^\infty((-1,0] \times B_1 \times \R^d)} + \|c\|_{L^\infty(Q_1)} + \|h\|_{L^\infty(Q_1)} \right).
\end{equation}
It is easy to see that
\[  \|\tilde f\|_{L^\infty((-1,0] \times B_1 \times \R^d)}  \leq \|\tilde f\|_{L^\infty((-1,0] \times B_1 \times B_2)}  \leq \|f\|_{L^\infty(Q_1)} + C \sup_{t,x} \int_1^2 \omega_0(r) r^{-1-2s} \dd r.\]
It only remains to prove that we have for any $v \in B_1$,
\[ |h(t,x,v)| \leq \int_{|v'|>3/2} f(t,x,v') K(t,x,v,v') \dd v' \leq \|f\|_{L^\infty(Q_1)} + C \int_{1/2}^\infty \omega_0(r) r^{-1-2s} \dd r.\]
Let us justify this inequality. Arguing as in Lemma~\ref{l:tail}, we have
  \begin{align*}
    \int_{|v'|>3/2} f(t,x,v') K(t,x,v,v') \dd v' & \leq \int_{3/2 < |v'| <7} f(t,x,v') K(t,x,v,v') \dd v' + \int_{|v'|>7} f(t,x,v') K(t,x,v,v') \dd v' \\
                                                 & \le \Lambda \|f(t,x,\cdot) \|_{L^\infty (B_7 \setminus B_{3/2} )} + \int_{|v'-v| > 6} f(t,x,v') K(t,x,v,v') \dd v'.
  \end{align*}
  Moreover,
  \begin{align*}
    \int_{|v'-v| > 6} f(t,x,v') K(t,x,v,v') \dd v' & \leq \int_{|v'-v|>6} \left\{ \inf_{3r/4 < |v'-v| < 3r/2} \omega_0 (r) \right\} K(t,x,v,v') \dd v' \\
    & \lesssim \int_4^{+\infty} \frac{\omega_0(r)}{r^{1+2s}} \dd r.
    \end{align*}
    We also have
    \[ \|f(t,x,\cdot)\|_{L^\infty (B_7 \setminus B_{3/2})} \lesssim \int_{3/4}^{14} \omega_0 (r) \dd r \lesssim \int_{1/2}^{+\infty} \omega_0 (r) r^{-1-2s} \dd r. \]
    This achieves the proof of the corollary.
\end{proof}

For convenience, we also state the scaled version of the previous result.
\begin{cor} \label{c:scaled-weak-harnack}
  Let $\LL$ be an integral operator corresponding to a kernel in the class $\K$ (as in Definition~\ref{d:class-of-kernels}) and 
  $f: (-R^{2s},0] \times B_{R^{1+2s}} \times \R^d$  be a function that solves the equation \eqref{e:kinetic-constant} in $Q_R$.
Then, the following estimate holds
\[ [f]_{C_\ell^\delta(Q_{R/2})} \leq C \left( R^{-\delta} \|f\|_{L^\infty(Q_R)} + R^{2s-\delta} \|c\|_{L^\infty(Q_R)}
    + R^{2s-\delta} \sup_{t \in (-R^{2s},0], x \in B_{R^{1+2s}}} \int_{R/2}^\infty \omega_0(t,x,r) r^{-1-2s} \dd r \right).\]
Here $\delta>0$ and $C$ are constants depending only on dimension and the parameters $\lambda$ and $\Lambda$ of Definition \ref{d:class-of-kernels}.
\end{cor}

\section{Liouville theorem}
\label{s:liouville}

This section is devoted to the statement and the proof of a theorem of Liouville type. 
\begin{thm}[Liouville] \label{t:liouville} Let
  $0 < \gamma < \min(1,2s)$ and $\alpha = \frac{2s}{1+2s} \gamma$.
  Assume that $\lfloor 2s+\alpha \rfloor < 2s+\alpha' < 2s + \alpha$
  and $ \alpha - \alpha'< \delta$  where $\delta$ is the constant
  from Theorem~\ref{t:weak-harnack}.

Let $f \in C_{\ell,loc}^{2s+\alpha'}((-\infty,0] \times \R^d \times \R^d)$ be a function that satisfies the following conditions.
\begin{itemize}
\item[(i)]  There is a constant $C_1 > 0$ such that for all  $R \geq 1$,
  \[
    [f]_{C_\ell^\beta(Q_R)} \leq C_1 R^{2s+\alpha-\beta}  \qquad \text{ for all } \beta \in [0,2s+\alpha'].
  \]
\item[(ii)] For any $\xi = (h,y,w) \in \R^{1+2d}$, with $h< 0$, we define
\( g(z) := f(\xi \circ z) - f(z)\) or equivalently, 
\( g(t,x,v) := f(t+h,x+y+ tw,v+w) - f(t,x,v).\)
Then, $g$ solves the equation
\begin{equation} \label{e:L-eq-for-g}
 g_t + v \cdot \nabla_x g = \LL g \text{ in } (-\infty,0] \times \R^d \times \R^d 
\end{equation}
where $\LL$ is the operator associated to some kernel $K \in \K$ as defined in \eqref{e:L}.
\end{itemize}
Then $f$ is a polynomial of kinetic degree  smaller than $2s+\alpha$.
\end{thm}
\begin{remark}
  Note that the assumption (i) ensures that the tails of $\LL g$ are
  integrable.  Indeed, let us take $\beta=\gamma+\eps < \min(1,2s)$
  for $\eps$ small. The assumption (i) tells us that
\begin{equation} \label{e:L-g1}
 |g(z)| \leq C \|z^{-1} \circ \xi \circ z\|^\beta \|z\|^{2s+\alpha-\beta} .
\end{equation}
Note that the condition $\beta<\min(1,2s)$ ensures that the polynomial $q_z$ in the definition of $[f]_{C_\ell^\alpha(z)}$ is the constant $q_z(\xi) = f(z)$.

Observe that for $z= (t,x,v)$ and $\xi = (h,y,w)$, we have
\begin{eqnarray}\label{e:commute}
  \| z^{-1} \circ \xi \circ z \| & =& \| (h,y+tw-h v,w)\| \leq \| \xi\| + |(tw-hv)|^{1/(1+2s)} \\
  \nonumber &\leq & C_\xi (1+ |t|+|v|)^{1/(1+2s)}.
\end{eqnarray}
Recalling that $\beta = \gamma +\eps$, we get

\begin{equation} \label{e:L-majorant}
\begin{aligned}
 g(t,x,v) &\leq C_{t,x,\xi} (1 + |v|)^{(\gamma+\eps)/(1+2s) + 2s + \alpha - \gamma - \eps}, \\
&=  C_{t,x,\xi} (1 + |v|)^{2s (1- \eps/(1+2s))} =: \omega(|v|).
\end{aligned}
\end{equation}

The operator $\LL g$ is well defined because this function $\omega$
suffices to bound the expression \eqref{e:L-computable}. The constant
$C_{t,x,\xi}$ depends on $t$, $x$, $\xi$, and the constant in the
assumption (i) with $\beta = \gamma+\eps$.

The tails of $\LL f$ may not be integrable, and therefore we can only
make sense of the equation for $g$, and not for $f$.
\end{remark}
\begin{remark}
  It is plausible that a version of this Liouville type result holds
  also for higher values of $\alpha$. In that case, for the equation
  \eqref{e:L-eq-for-g} to make sense, we would have to make $g$ a
  higher order incremental quotient of $f$.
\end{remark}

We start with a simpler Liouville type result that is a consequence of
the H\"older estimate contained in Theorem \ref{t:weak-harnack}.
\begin{lemma}[Liouville] \label{l:liouville-baby} Let $\delta$ be the
  constant from Theorem \ref{t:weak-harnack}. Assume $\beta<\delta$
  and $f$ is a solution to \eqref{e:kinetic-constant} in
  $(-\infty,0] \times \R^d \times \R^d$ with $c=0$. Assume further
  that for all $R \geq 1$,
\[ \|f\|_{L^\infty(Q_R)} \leq C R^\beta,\]
then $f$ is constant. 
\end{lemma}
\begin{proof}
  We apply Corollary \ref{c:scaled-weak-harnack} in $Q_R$ and make
  $R \to 0$. From our assumption on the growth of $f$, for all
  $(t,x) \in (-R^{2s},0] \times B_{R^{1+2s}}$ we have that
  $\omega(t,x,r) \lesssim r^\beta$. Thus, we have
\[ \int_{R/2}^\infty \omega(t,x,r) r^{-1-2s} \dd r \lesssim R^{\beta-2s}.\]
Then, Corollary \ref{c:scaled-weak-harnack} tells us that
\[ [f]_{C_\ell^\delta(Q_{R/2})} \leq C R^{\beta-\delta}.\] Taking
$R \to \infty$, the semi-norm $[f]_{C_\ell^\delta(Q_R)}$ converges to zero,
and then the function $f$ must be constant.
\end{proof}

\begin{proof}[Proof of Theorem~\ref{t:liouville}]
  We first claim that it is enough to prove the result assuming that
  $f \in C^\infty$. Indeed, if $f$ is less regular, we can mollify it
  respecting the Lie group structure then apply the result to the
  approximate function and pass to the limit.

  The remainder of the proof proceeds in several steps. \medskip
  
\noindent{\sc Step 1: $f$ is constant in $x$.}
Let $y \in \R^d$ and $g(t,x,v) := f(t,x-y,v) - f(t,x,v)$. We apply the assumption (i) with $\beta = 2s+\alpha'$. Note that $2s+\alpha' < 1+2s$ by assumption. We get
\begin{equation} \label{e:L-fx1}
 \|g \|_{L^\infty(Q_R)} \leq C_1 |y|^{(2s+\alpha') / (1+2s)} R^{\alpha-\alpha'}.
\end{equation}
Since we assume that $\alpha - \alpha' < \delta$, then we can apply Lemma \ref{l:liouville-baby} and we get that $g$ is constant. Therefore, $f$ must be of the form
\[ f(t,x,v) = a \cdot x + f(t,v),\]
for some constant $a \in \R^d$. However, the assumption (i)  tells us that for all $R \geq 1$,
\[ \osc_{Q_R}  f = [f]_{C^0 (Q_R)} \leq C_1 R^{2s+\alpha}.\]
This is only possible if $a = 0$ (recall that $a \cdot x$ is a polynomial of order $1+2s > 2s+\alpha$). Thus, $f$ is independent of $x$ and from now on we write $f = f(t,v)$.
\medskip

\noindent{\sc Step 2: $f_t$ is constant}.
Observe that the kinetic order of $\partial_t$ is $2s$. Therefore, $f_t$ is well defined since $f \in C_{\ell,loc}^{2s+\alpha'}$. Moreover, from the assumption (i) and Lemma~\ref{l:holder-derivatives},
we deduce that,
\begin{equation} \label{e:L-ft1}
  \begin{cases}
    \| f_t \|_{L^\infty (Q_R)} \le C_1 R^{\alpha} \\
    [f_t ]_{C_\ell^{\tilde{\beta}}(Q_R)} \leq C_1 R^{\alpha-\tilde{\beta}} \qquad \text{ for all } \tilde{\beta} \in (0,\alpha'].
    \end{cases}
\end{equation}
Since $f_t(t,v) = \lim_{h \to 0} (f(t+h,v) - f(t,v)))/h$, using (ii) we deduce that
\[ \partial_t(f_t) = \LL f_t.\]
We omitted the term $v \cdot \nabla_x f_t$ because it is identically zero.

Using the invariance of the equation by the Lie group action and the
fact that $f$ is independent of $x$, we have that for any
$(h,w) \in [0,\infty) \times \R^d$, the function
\[ g := f_t(t+h,v+w) - f_t(t,v),\]
also solves
\[ g_t = \LL g.\]
Because of \eqref{e:L-ft1}, with $\tilde{\beta}=\alpha'$, we get that
\begin{align*} 
 \|g\|_{L^\infty(Q_R)} &\lesssim C_1 \| (h,0,w)\| R^{\alpha-\alpha'}.
\end{align*}
Thus, we obtain that $g$ is constant applying Lemma
\ref{l:liouville-baby}. Therefore, $f_t$ must be of the form
$f_t(t,v) = a\cdot v + bt + c$. However, \eqref{e:L-ft1}
implies that the kinetic degree of $f_t$ cannot be
more than $\alpha$, and therefore $f_t$ is a constant.

Since $f$ is independent of $x$ and $f_t$ is constant, then $f$ has
the form $f = at + \tilde f(v)$ for some constant $a$. The function
$\tilde f$ satisfies
\[ a = \LL \tilde f.\]

We are left to prove that $f$ is a polynomial in $v$. 
\medskip

\noindent{\sc Step 3: $f$ is a polynomial in $v$.} The third step is
also divided into three cases depending the integer part of
$2s+\alpha$. Indeed, the maximum
number of terms in the polynomial $\tilde f(v)$ will depend of
$2s+\alpha$ belonging to the three possible ranges $(0,1)$, $(1,2)$ or
$(2,3)$ (recall that $2s+\alpha$ is not an integer).

Let us start by assuming that $2s+\alpha \in (0,1)$. Given any
$w \in \R^d$, we set $g(v) := \tilde f(v+w) - \tilde f(v)$. Applying
the assumption (i), with $\beta = 2s + \alpha'$ we get
\[  \|g\|_{L^\infty(Q_R)} \leq C_1 |w|^{\alpha'} R^{\alpha-\alpha'}.\]
Moreover, $g$ solves
\[ 0 = \LL g.\] We apply Lemma \ref{l:liouville-baby} right away. We
deduce that $g$ is constant for any $w \in \R^d$. Therefore $\tilde f$
has the form $\tilde f(v) = b \cdot v + c$. However, the assumption
(i) with $\beta=0$ implies in this case that $b = 0$, so $\tilde f$
must be constant. Therefore, in the case $2s+\alpha \in (0,1)$ we
conclude that $f(t,x,v) = at + c$ for some constants $a$ and
$c$. Thus, $f$ is a polynomial of degree at most $2s$.

In the case $2s+\alpha \in (1,2)$ the function $f$ must be
differentiable in $v$ because of the assumption (i) applied with
$\beta \in (1,2s+\alpha')$. Thus, if we let
$f_j = \partial_{v_j} \tilde f$ we get that
\begin{itemize}
\item[(i')] There is a $C_1 > 0$ such that for all $\beta \in [0,2s+\alpha'-1]$ and $R \geq 1$,
\[ [f_j]_{C^\beta(B_R)} \leq C_1 R^{2s+\alpha-1-\beta}.\]
\item[(ii')] For any $w \in \R^{d}$, we define
\[ g(v) := f_j(v+w) - f_j(v).\]
Then, $g$ solves
\[ 0 = \LL g.\]
\end{itemize}

Therefore, we repeat the proof of the case $2s+\alpha \in (0,1)$ for
$f_j$ instead of $\tilde f$ and get that each partial derivative $f_j$
is constant. Therefore, in this case $\tilde f$ must be an affine
function.

Likewise, in the case $2s+\alpha \in (2,3)$, we apply the argument for
$2s+\alpha \in [1,2)$ to each partial derivative
$f_j = \partial_{v_j} \tilde f$. In this case we obtain that each
$f_j$ is affine, and therefore $\tilde f$ must be a polynomial in $v$
of degree at most $2$.
\end{proof}

\section{Blowup argument}
\label{s:blowup}

In this section, we prove that the H\"older exponent of a solution of
a linear equation of the form \eqref{e:main} can be improved;
moreover, the improvement is quantitative. The result is proved by
blowup and compactness. It is first proved for equation with
``constant coefficients'' (Proposition~\ref{p:blowup}) and then proved
in the general case
(Proposition~\ref{p:localized-improvement-of-holder}). 
\begin{prop}[Improvement by blow up for ``constant coefficients''] \label{p:blowup}
Let $\alpha$ and $\alpha'$ be as in Theorem~\ref{t:liouville}. Assume that 
\[ f_t + v \cdot \nabla_x f - \LL f = c_0(t,x,v) \ \text{ in } Q_1,\]
for some function $c_0 \in C_\ell^\alpha(Q_1)$ and some kernel $K \in \K$. Then
\[ [f]_{C_\ell^{2s+\alpha}(Q_{1/2})} \leq C \left( \|c_0\|_{C_\ell^\alpha(Q_1)} + \|f\|_{C_\ell^{2s+\alpha'}((-1,0] \times B_1 \times \R^d)} \right)\]
where $C$ only depends on $d$, $s$, $\alpha$, $\alpha'$, $\lambda$ and $\Lambda$.
\end{prop}
Before proving the proposition, we state a lemma corresponding to
 \cite[Claim~3.2]{serra1}. Its adaptation to kinetic H\"older spaces
is straight forward.
\begin{lemma}[From $C_\ell^{\beta'}$ to $C_\ell^\beta$ with $\beta > \beta'$] \label{l:holder-exponent-improvement} Let
  $0 < \beta' < \beta$. Let $\nu$ be the maximum number in
  $\mathbb N + 2s \mathbb N$ such that $\nu < \beta$. Assume that
  $\nu < \beta' < \beta$. Let $f$ be a continuous function in
  $C_\ell^{\beta'}((-\infty,0] \times \R^d \times \R^d)$ and let $C_0$ be
  such that
  \[ \sup_{r > 0} \sup_{z \in Q_{1/2}} r^{\beta'-\beta}
    [f]_{C_\ell^{\beta'}(Q_r(z))} \leq C_0.\] Then
  $[f]_{C_\ell^\beta(Q_{1/2})} \leq C_0$. Moreover, if
  $\beta \notin \mathbb N + 2s \mathbb N$ (and assuming $f$ smooth),
  the supremum is attained at some $r >0 $ and
  $z \in \overline{Q_{1/2}}$.
\end{lemma}
We can now turn to the proof of Proposition~\ref{p:blowup}. 
\begin{proof}[Proof of Proposition~\ref{p:blowup}]
  Without loss of generality, we normalize the problem so that
  \[ \|c_0\|_{C_\ell^\alpha(Q_1)} \leq 1 \quad \text{ and } \quad
    \|f\|_{C_\ell^{2s+\alpha'}((-1,0] \times B_1 \times \R^d)} \leq 1.\]
  Under these conditions, we need to prove that
  $\|f\|_{C_\ell^{2s+\alpha}(Q_{1/2})} \lesssim 1$.  We proceed by
  contradiction.  Assuming the opposite, there would exist sequences
  $f_j$, $K_j \in \K$ and $c_j$ such that,
\begin{align}
\label{e:cj} \|c_j\|_{C_\ell^\alpha(Q_1)} \leq 1, \\
\label{e:fj} \|f_j\|_{C_\ell^{2s+\alpha'}((-1,0] \times B_1 \times \R^d)} \leq 1, \\[2ex]
\label{e:ej} \partial_t f_j + v \cdot \nabla_x f_j - \LL_j f_j = c_j \text{ in } Q_1, \\
\label{e:Fj}  \left( \sup_{r > 0} \sup_{z \in Q_{1/2}} r^{\alpha'-\alpha} [f_j]_{C_\ell^{2s+\alpha'}(Q_r(z))} \right) \nearrow +\infty \text{ as } j \to \infty.
\end{align}
The last property holds since we cannot apply Lemma~\ref{l:holder-exponent-improvement} uniformly.
Thanks to Lemma~\ref{l:holder-exponent-improvement}, there exists $r_j>0$ and $z_j \in \overline{Q_{1/2}}$ such that
\[ F_j := \sup_{r > 0} \sup_{z \in Q_{1/2}} r^{\alpha'-\alpha} [f_j]_{C_\ell^{2s+\alpha'}(Q_r(z))} =r_j^{\alpha'-\alpha} [f_j]_{C_\ell^{2s+\alpha'}(Q_{r_j}(z_j))} .\]
In particular, $F_j \to +\infty$ and $r_j \to 0^+$.

We define
\[ \tilde f_j(z) := \frac{ (f_j-q_j)(z_j \circ S_{r_j} (z))  }{r_j^{2s+\alpha} F_j}\]
where $q_j$ denotes the polynomial expansion of $f_j$ at $z_j$ as in the definition of $C_\ell^{2s+\alpha'}(Q_{r_j} (z_j))$.
This sequence $\tilde f_j$ satisfies the following properties.
\begin{itemize}
\item Since $F_j = r_j^{\alpha'-\alpha} [f_j]_{C_\ell^{2s+\alpha'}(Q_{r_j}(z_j))}$, we have
\begin{equation} \label{e:b1-normalization}
 [\tilde f_j]_{C_\ell^{2s+\alpha'}(Q_1)} = 1.
\end{equation}
\item Since $F_j \geq r^{\alpha'-\alpha} [f_j]_{C_\ell^{2s+\alpha'}(Q_{r}(z_j))}$, for all $r > r_j$,  we have
\[ [\tilde f_j]_{C_\ell^{2s+\alpha'}(Q_R)} \leq R^{\alpha-\alpha'} \ \text{ for all } R\geq 1.\]
\item Because of the substraction of the polynomial expansion $q_j$, we also have
\[ \|\tilde f_j\|_{L^\infty(Q_R)} \leq R^{2s+\alpha} \qquad \text{ for all } R \geq 1.\]
\item By interpolation
  (Proposition~\ref{p:adimensional-interpolation}) between the last
  two items, we deduce for all $\beta \in (0,2s+\alpha']$,
\begin{equation} \label{e:b1-growth}
 [\tilde f_j]_{C_\ell^\beta(Q_R)} \lesssim R^{2s+\alpha-\beta} \qquad \text{ for all } R \geq 1.
\end{equation}
\end{itemize}

For each choice of $\xi  \in (-\infty,0] \times \R^d \times \R^d$, we define
\[ g_j(z)  := \tilde f_j(\xi \circ z ) - \tilde f_j(z).\]
Condition~\eqref{e:b1-growth} allows us to bound the growth of
$g_j(t,x,v)$ as $|v| \to \infty$ by a single majorant function
$\omega(|v|)$ as in \eqref{e:L-majorant}.

Because of \eqref{e:ej}, each function $g_j$ satisfies for $j$ large enough (see the choice of $R_j$ below) the equation
\begin{equation} \label{e:b1-eq-for-g} \partial_t g_j + v \cdot
  \nabla_x g_j - \tilde \LL_j g_j = \tilde c_j(\xi \circ z) - \tilde
  c_j(z)  \text{ in } Q_{R_j} (0),
\end{equation}
where the operator $\tilde \LL_j$ corresponds to a scaled kernel $\tilde K_j$ and the source term $c_j$ is scaled too, 
\begin{align*}
\tilde K_j(w) &= r_j^{-2s} K_j (r_j w), \\
\tilde c_j(z) &= r_j^{-\alpha} F_j^{-1} c_j(z_j \circ S_{r_j} (z)).
\end{align*}
In particular, we choose the radius $R_j \to +\infty$ such that
  $Q_{2 R_j r_j} (z_j) \subset Q_1$ and $\xi \circ z \in Q_{2R_j} (0)$
  for $z \in Q_{R_j} (0)$.

Because of \eqref{e:cj}, it is straight forward to verify that,
\[ \| \tilde c_j(\xi \circ z) - \tilde c_j(z) \|_{L^\infty(Q_R)} \lesssim  F_j^{-1} d_\ell (\xi \circ z, z)^\alpha \lesssim F_j^{-1}.\]
Therefore, the right hand side of \eqref{e:b1-eq-for-g} converges to zero over any compact set.

We observe that all the kernels $\tilde K_j$ belong to the class
$\K$. Applying Lemma~\ref{l:weak-limit-compactness}, they converge
weak-$\ast$ to a kernel $K_\infty$ up to a subsequence. Because of
Lemma~\ref{l:weak-limit-stay-in-class}, $K_\infty \in \K$.

We pick $\eps>0$ so that
$\lfloor 2s+\alpha \rfloor < 2s+\eps < 2s+\alpha'$.  Because of
Arzela-Ascoli theorem, we take a subsequence so that $\tilde f_j$
converges to some function $f: (-\infty,0]\times \R^d \times \R^d$
locally in $C_\ell^{2s+\eps}$.  This function $f_\infty$ also satisfies
\eqref{e:b1-growth}.

Since all the $\tilde f_j$ are controlled by a single majorizer
$\omega$, we can apply Lemma~\ref{l:weak-limit-stability} and we have
that the corresponding function $g_\infty = \lim \tilde g_j$ solves
the limit equation from \eqref{e:b1-eq-for-g},
\[ \partial_t g_\infty + v \cdot \nabla_x g_\infty - \LL_\infty g_\infty = 0 \text{ in } (-\infty,0] \times \R^d \times \R^d.\]

We are then able to apply the Liouville theorem~\ref{t:liouville} and
get that $g_\infty$ is a polynomial. However, we subtracted the polynomial
expansion of $f_j$ in the definition of $\tilde{f}_j$, forcing $\tilde f_j$ to
have a vanishing polynomial expansion at $0$ of order up to
$2s+\alpha'$. Since $\tilde f_j \to f_\infty$ in $C_\ell^{2s+\eps}$ and
$2s+\eps > \lfloor 2s+\alpha \rfloor$, then all derivatives of
$f_\infty$ at the origin must be zero and therefore $f_\infty = 0$.

This contradicts \eqref{e:b1-normalization} and  the proof is complete.
\end{proof}
\begin{remark} \label{rem:ext-lim} We would like to emphasize the
  importance of the majorant function $\omega$. In order to do so, we
  point out that it is possible to follow the same outline of the
  proof by blowup to obtain the result of
  Lemma~\ref{l:holder-exponent-improvement} under more general
  assumptions, at the expense of a more complicated proof. Again, we
  should not overlook the importance of the majorant function
  $\omega$. For example, with a more complicated function $\omega$ it
  is possible to derive an estimate of the form
\begin{equation} \label{e:hw1}
  \|f\|_{C_\ell^{2s+\alpha}(Q_{1/2})} \leq C \left( \|f\|_{C_\ell^{2s+\alpha'}(Q_1)}
    + \|f\|_{C_\ell^\gamma((-1,0] \times B_1 \times \R^d)} + \|c\|_{C_\ell^\alpha((-1,0] \times B_1 \times \R^d)} \right),
\end{equation}
provided that $\alpha \leq 2s \gamma/(1+2s)$.
However, the following estimate is not true for any $\gamma<\alpha$, even in the elliptic case
\begin{equation} \label{e:hw2}
  \|f\|_{C_\ell^{2s+\alpha}(Q_{1/2})} \leq C \left( \|f\|_{C_\ell^{2s+\alpha'}(Q_1)}
    + \|f\|_{C_\ell^\gamma((-1,0] \times B_1 \times \R^d)} + \|c\|_{C_\ell^\alpha((-1,0] \times B_1 \times \R^d)} \right).
\end{equation}
If we tried to reproduce the proof of Lemma
\ref{l:holder-exponent-improvement} for this last inequality, we would
still have scaled functions $\tilde f_j$ that converge locally
uniformly to a function satisfying the assumption (i) in Liouville's
theorem~\ref{t:liouville}. We can still construct the functions
$\tilde g_j$ that converge uniformly to a function $g_\infty$. But
unless we have an appropriate control on the tails of the original
functions $f_j$, we cannot conclude that the limit function $g_\infty$
will satisfy any equation.
\end{remark}

We now extend the blowup lemma to the case of ``variable
coefficients''. In order to do so, we are going to use
Assumption~\ref{a:K-holder} depending on a constant $A_0>0$.
\begin{prop}[Improvement by blowup for variable
  coefficients] \label{p:localized-improvement-of-holder} Let
  $\alpha,\alpha'$ as in Theorem~\ref{t:liouville}. Under the
  assumptions of Theorem~\ref{t:main}, we have
\[
  [f]_{C_\ell^{2s+\alpha}(Q_{1/4})} \leq C \left( \|f\|_{C_\ell^{2s+\alpha'}(Q_1)}
    + A_0 \|f\|_{C_\ell^{2s+\alpha}(Q_1)}  + (1+A_0) \|f\|_{C_\ell^\gamma((-1,0] \times B_1 \times \R^d)} + \|c_0\|_{C_\ell^\alpha(Q_1)} \right),
\]
where the constant $C$ depends on dimension, $s$, $\alpha$, $\alpha'$, $\lambda$ and $\Lambda$.
\end{prop}
\begin{proof}
Let $\eta \in C^\infty_c((-1,0] \times B_1 \times \R^d)$. Assume that $\eta = 1$ in $Q_{3/4}$ and that $\eta(z) = 0$ whenever $z \notin Q_1$.

Let $\tilde f(z) = \eta(z) f(z)$. Obviously, we have
\[ \|\eta f\|_{C_\ell^{2s+\alpha'}((-1,0] \times B_1 \times \R^d)} \leq C \|f\|_{C_\ell^{2s+\alpha'}(Q_1)}.\]
We want to estimate the right hand side that would make $\tilde f$ satisfy an equation as in \eqref{e:kinetic-constant}. We \emph{freeze coefficients} first:
\[ K_0(w) = K(0,0,0,w).\]
and let $\LL_0$ be the corresponding integro-differential operator. A straight-forward computation shows that for all $z \in Q_{1/2}$,
\[ \tilde f_t + v \cdot \nabla_x \tilde f - \LL_0 \tilde f = c(z) + A(z) + B(z).\]
where \label{p:ab}
\begin{align*} 
 A(z) :&= \int_{\R^d} (f(t,x,v+w) - f(t,x,v)) [ K_z(w) - K_0(w)] \dd w, \\
B(z) :&=  \int_{\R^d} (\eta(t,x,v+w)- \eta(t,x,v) ) f(t,x,v+w) K_0(w) \dd w.
\end{align*}

\begin{lemma}
\label{l:estim-A}
\(  [A(z)]_{C_\ell^\alpha(Q_{1/2})} \lesssim A_0 \left( \|f\|_{C_\ell^{2s+\alpha}(Q_1)}
    + \|f\|_{C_\ell^\gamma((-1,0]\times B_1 \times \R^d)} \right).
\)
\end{lemma}
\begin{proof}
  We write $A(z_1) - A(z_2) = I_1 + I_2$ with
\begin{align*}
  I_1 & = \int (f(z_2 \circ (0,0,w))  - f(z_2)) (K_{z_1} (w) - K_{z_2} (w)) \dd w, \\
  I_2 & = \int (f(z_1 \circ (0,0,w)) -f(z_1) -f(z_2 \circ (0,0,w)) + f(z_2)) (K_{z_1} (w) - K_0 (w)) \dd w.
\end{align*}
As far as $I_1$ is concerned, we write $I_1 = I_1^{\text{out}} + I_1^{\text{in}}$ with
\begin{align*}
  I_1^{\text{out}} & = \int_{|w|\ge 1} (f(z_2 \circ (0,0,w))  - f(z_2)) (K_{z_1} (w) - K_{z_2} (w)) \dd w, \\
  I_1^{\text{in}} & = \int_{|w| \le 1} (f(z_2 \circ (0,0,w))  - f(z_2)) (K_{z_1} (w) - K_{z_2} (w)) \dd w.
\end{align*}
For the nonsingular part of $I_1$, we simply write
\begin{align*}
|  I_1^{\text{out}}| &\le \|f\|_{L^\infty((-1,0] \times B_1 \times \R^d)} \int_{|w| \ge 1} |K_{z_1} (w) - K_{z_2} (w)| \dd w\\
                     & \lesssim A_0 \|f\|_{L^\infty((-1,0] \times B_1 \times \R^d)} d_\ell (z_1,z_2)^\alpha.
\end{align*}
where we used \eqref{e:a0-infty}, which is a consequence of \eqref{a:K-holder}. 

We now turn to $I_1^{\text{in}}$. We write
\begin{align*}
  |  I_1^{\text{in}}|  \le &  \int_{|w| \le 1} |f(z_2 \circ (0,0,w))  - p_{z_2} ((0,0,w))| |K_{z_1} (w) - K_{z_2} (w)| \dd w \\
                           & + \frac12 \int_{|w| \le 1} | p_{z_2} ((0,0,w)) + p_{z_2} ((0,0,-w))- f(z_2)  | |K_{z_1} (w) - K_{z_2} (w)| \dd w \\
  \lesssim     & [f]_{C_\ell^{2s+\alpha}} \int_{|w| \le 1} |w|^{2s+\alpha} |K_{z_1} (w) - K_{z_2} (w)| \dd w \\
                           & + \|f\|_{C_\ell^{2s+\alpha}}  \int_{|w| \le 1} |w|^{2} |K_{z_1} (w) - K_{z_2} (w)| \dd w  \qquad \qquad \text{(only relevant if $2s+\alpha>2$)}\\
  \lesssim & A_0 \|f\|_{C_\ell^{2s+\alpha}} d_\ell (z_1,z_2)^\alpha.
\end{align*}
We used \eqref{e:a0-2s+alpha} which is a consequence of \eqref{a:K-holder}.

We now estimate $I_2$ thanks to Lemma~\ref{l:estimA0-loc}. This achieves the proof of the lemma.   
\end{proof}

We now estimate the $C_\ell^\alpha(Q_{1/2})$ norm of $B$.
\begin{lemma}
  \label{l:estim-B}
\(  [B]_{C_\ell^\alpha(Q_{1/2})} \lesssim \|f \|_{C_\ell^\gamma ((-1,0] \times B_1 \times \R^d)}.\)
\end{lemma}
\begin{proof}
For $z_1, z_2 \in Q_{1/2}$, we compute,
\[ B(z_1) - B(z_2) = J_1 + J_2\]
with
\[
  \begin{cases}
    J_1 & = \int_{|w| > 1/4} (\eta(z_1\circ (0,0,w)) -  \eta(z_1)-\eta(z_2\circ (0,0,w)) + \eta (z_2)) f(z_1 \circ (0,0,w)) K_0(w) \dd w, \\
    J_2 & = \int_{|w| > 1/4} (\eta(z_2\circ (0,0,w)) - \eta (z_2)) (f(z_1 \circ (0,0,w))-f(z_1 \circ (0,0,w)))  K_0(w) \dd w.
  \end{cases}
\]

We turn to estimate $J_1$. Since $\eta$ is smooth, and in particular $C_\ell^\gamma$, we can apply Lemma~\ref{l:estimA0-loc} and get
\begin{align*}
|J_1| & \lesssim \|f \|_{L^\infty ((-1,0] \times B_1 \times \R^d)} d_\ell (z_1,z_2)^\alpha.
\end{align*}

As far as $J_2$ is concerned, we get
\begin{align*} 
 |J_2| &\leq 2 \|\eta\|_\infty [f]_{C_\ell^\gamma} \int_{|w| > 1/4}   d_\ell(z_1 \circ (0,0,w), z_2 \circ (0,0,w))^\gamma K_0(w) \dd w, \\
\intertext{using that $  \alpha = 2s \gamma / (1+2s)$,}
&\lesssim [f]_{C_\ell^\gamma} d_\ell(z_1,z_2)^{\frac{2s}{1+2s} \gamma} \lesssim [f]_{C_\ell^\gamma} d_\ell(z_1,z_2)^\alpha.
\end{align*}
This achieves the proof of the estimate for $B$.
\end{proof}

Thanks to Lemmas~\ref{l:estim-A} and \ref{l:estim-B}, we conclude the proof of
Proposition~\ref{p:localized-improvement-of-holder} by applying
Proposition~\ref{p:blowup} to $\tilde f$ (with $Q_1$ replaced with
$Q_{\frac12}$).
\end{proof}
\begin{proof} [Proof of Theorem \ref{t:main}]
Without loss of generality, we perform an initial scaling to make sure that the constant $A_0$ is small (it is a subcritical parameter).

We will next prove the slightly stronger estimate
\begin{equation} \label{e:final-estimate}
 \|f\|_{C_{\ell,\ast}^{2s+\alpha}(Q_1)} \lesssim \|f\|_{C_\ell^\gamma([-1,0] \times B_1 \times \R^d)} + \|c\|_{C_{\ell,\ast}^\alpha(Q_1)} .
\end{equation}

Scaling and translating the estimate from
Proposition~\ref{p:localized-improvement-of-holder}, we get that for
any cylinder $Q_\rho(z_0) \subset Q_1$,
\begin{align*}
  \rho^{2s+\alpha} [f]_{C_\ell^{2s+\alpha}(Q_{\rho/4}(z_0))} \lesssim & \rho^{2s+\alpha'} [f]_{C_\ell^{2s+\alpha'}(Q_\rho(z_0))} 
                                                                   + (\rho^{\alpha} A_0) \rho^{2s+\alpha} [f]_{C_\ell^{2s+\alpha}(Q_\rho (z_0))}  \\
                                                                 &  + \rho^\gamma [f]_{C_\ell^\gamma((-1,0] \times B_1 \times \R^d)} 
                                                                   + \|f\|_{L^\infty((-1,0] \times B_1 \times \R^d)}  \\
                                                                 &+ \rho^{2s+\alpha} [c_0]_{C_\ell^\alpha(Q_\rho (z_0))}   
                                                                    +  \rho^{2s}\|c_0\|_{L^\infty(Q_\rho (z_0))} .
\end{align*}
For any $z_0$ in $Q_1$, we choose $\rho = d_\ell(z_0,\partial Q_1) < 1$ so that $Q_\rho(z_0) \subset Q_1$.

Note that $A_0 \rho^\alpha \leq A_0$ and $\rho^\gamma [f]_{C_\ell^\gamma((-1,0] \times B_1 \times \R^d)} \leq [f]_{C_\ell^\gamma((-1,0] \times B_1 \times \R^d)}$. We get
\[
  \|f\|_{C_{\ell,\ast}^{2s+\alpha}(Q_1)} \lesssim  [f]_{C_{\ell,\ast}^{2s+\alpha'}(Q_1)} + A_0 \|f\|_{C_{\ell,\ast}^{2s+\alpha}(Q_1)}
  +\|f\|_{C_\ell^\gamma([-1,0] \times B_1 \times \R^d)} + \|c\|_{C_{\ell,\ast}^\alpha(Q_1)} .
\]
We use the interpolation result of
Proposition~\ref{p:adimensional-interpolation} (see
Remark~\ref{r:inter-eps}) to get
$ [f]_{C_{\ell,\ast}^{2s+\alpha'}(Q_1)} \leq \eps
[f]_{C_{\ell,\ast}^{2s+\alpha}(Q_1)} + C_\eps \|f\|_{L^\infty(Q_1)}$.  This
achieves the proof of the main theorem.
\end{proof}

\bibliographystyle{plain}
\bibliography{skie}
\end{document}